\documentclass[12pt,leqno]{amsart}
\usepackage{latexsym,amsmath,amssymb}

\title[Weak contact equations]{Weak contact equations for mappings into Heisenberg groups}

\author{Zolt\'an M. Balogh, Piotr Haj{\l}asz, and Kevin Wildrick}

\address{Z.M.\ Balogh: Mathematisches Institut, Universit\"at Bern, Sidlerstrasse 5, 3012 Bern, Switzerland ({\tt balogh.zoltan@math.unibe.ch})}

\address{P.\ Haj{\l}asz: Department of Mathematics, University of Pittsburgh, 301
  Thackeray Hall, Pittsburgh, PA 15260, USA, {\tt hajlasz@pitt.edu}}

\address{K. Wildrick: Mathematisches Institut, Universit\"at Bern, Sidlerstrasse 5, 3012 Bern, Switzerland ({\tt kevin.wildrick@math.unibe.ch})}

\thanks{Z.M.B.\ and K.W.\ were supported by the Swiss National Science Foundation, 
European Research Council Project CG-DICE, and the European Science Council 
Project HCAA. P.H.\ was supported by NSF grant DMS-1161425.}

\def\rank{{\rm rank\,}}
\newcommand{\cJ}{\mathcal J}
\newcommand{\boldg}{{\mathbf g}}

\def\eps{\varepsilon}

\def\vi{\varphi}

\def\M{{\mathcal M}}
\def\H{{\mathcal H}}

\newtheorem{theorem}{Theorem}
\newtheorem{lemma}[theorem]{Lemma}
\newtheorem{corollary}[theorem]{Corollary}
\newtheorem{proposition}[theorem]{Proposition}

\def\diam{{\rm diam\,}}
\def\dist{{\rm dist\,}}

\def\lip{{\rm Lip\,}}

\def\loc{{\rm loc}}

\theoremstyle{definition}
\newtheorem{remark}[theorem]{Remark}

\newtheorem{example}[theorem]{Example}

\newcommand{\barint}{
\rule[.036in]{.12in}{.009in}\kern-.16in \displaystyle\int }

\newcommand{\barcal}{\mbox{$ \rule[.036in]{.11in}{.007in}\kern-.128in\int $}}

\def\eqn#1$$#2$${\begin{equation}\label#1#2\end{equation}}
\def\vint#1_#2{-\kern-#1pt\int_{#2}}

\newcommand{\bbbr}{\mathbb R}

\newcommand{\bbbh}{\mathbb H}
\newcommand{\bbbc}{\mathbb C}

\newcommand{\subeq}{\subseteq}
\newcommand{\inv}{^{-1}}

\def\ap{\operatorname{ap}}
\def\wk{\operatorname{wk}}
\def\loc{\operatorname{loc}}
\def\diam{\operatorname{diam}}

\def\dist{\operatorname{dist}}
\def\im{\operatorname{im}}

\def\mvint_#1{\mathchoice
          {\mathop{\vrule width 6pt height 3 pt depth -2.5pt
                  \kern -8pt \intop}\nolimits_{\kern -3pt #1}}%
          {\mathop{\vrule width 5pt height 3 pt depth -2.6pt
                  \kern -6pt \intop}\nolimits_{#1}}%
          {\mathop{\vrule width 5pt height 3 pt depth -2.6pt
                  \kern -6pt \intop}\nolimits_{#1}}%
          {\mathop{\vrule width 5pt height 3 pt depth -2.6pt
                  \kern -6pt \intop}\nolimits_{#1}}}

\numberwithin{theorem}{section} \numberwithin{equation}{section}

\begin{document}

\sloppy


\begin{abstract}
Let $k>n$ be positive integers. We consider mappings from a subset of $\bbbr^k$ to the Heisenberg 
group $\bbbh^n$ with a variety of metric properties, each of which imply that the mapping in question 
satisfies some weak form of the contact equation arising from the sub-Riemannian structure of the Heisenberg 
group. We illustrate a new geometric technique that shows directly how the weak contact equation greatly 
restricts the behavior of the mappings. In particular, we provide a new and elementary proof of the fact that 
the Heisenberg group $\bbbh^n$ is purely
$k$-unrectifiable. We also prove that for an open set $\Omega \subset \bbbr^{k}$, the rank of the weak 
derivative of a weakly contact mapping in the Sobolev space $W^{1,1}_{\loc}(\Omega;\bbbr^{2n+1})$
is bounded by $n$ almost everywhere, answering a question of Magnani. Finally
we prove that if $f \colon \Omega \to \bbbh^n$ is $\alpha$-H\"older continuous, $\alpha>1/2$, and locally 
Lipschitz when considered as a mapping into $\bbbr^{2n+1}$, then $f$ cannot be injective.  
This result is related to a conjecture of Gromov.
\end{abstract}

\maketitle

\section{Introduction}
\label{introduction}

Let $n$ be a positive integer. The Heisenberg group $\bbbh^n$ is a non-commutative Lie group 
structure on $\bbbc^n\times\bbbr=\bbbr^{2n+1}$
that plays an important role in a variety of areas, including mathematical physics, complex analysis, 
hyperbolic geometry, and control theory. It is also a key example in the theory of analysis on metric spaces.
Using the notation
$$
p=(z,t) = (z_1,\ldots,z_n,t) = (x,y,t)=(x_1,y_1,\ldots,x_n,y_n,t)
$$
for a point of $\bbbh^n$, the Lie algebra of $\bbbh^n$ is defined by the basis of left-invariant vector fields 
\begin{equation}
\label{XY}
X_j=\frac{\partial}{\partial x_j} + 2y_j\frac{\partial}{\partial t},\
Y_j=\frac{\partial}{\partial y_j}-2x_j\frac{\partial}{\partial t},\ j=1,\ldots,n, \
\mbox{and}\
T=\frac{\partial}{\partial t}\, .
\end{equation}
The Heisenberg group is equipped with the non-integrable
{\em horizontal distribution} $H\bbbh^n$, which is defined at every
point $p\in\bbbh^n$ by
$$
H_p\bbbh^n={\rm span}\, \{ X_1(p),\ldots,X_n(p),Y_1(p),\ldots,Y_n(p)\}.
$$
This distribution coincides with the kernel at $p$ of the \emph{standard contact form}
\begin{equation}
\label{the-alpha-form}
\alpha = dt + 2 \sum_{j=1}^n (x_j \, dy_j - y_j \, dx_j)
\end{equation}
on $\bbbr^{2n+1}$. 
The horizontal distribution naturally induces a length metric on $\bbbh^n$, called the Carnot-Carath\'eodory metric, by considering only curves that are almost everywhere tangent to the distribution; 
see Section \ref{heisenberg} for a more precise description. While its topological dimension is $2n+1$, 
the Heisenberg group has Hausdorff dimension $2n+2$ when equipped with this metric. 

In this paper, we consider mappings from subsets of the Euclidean space $\bbbr^k$ to the Heisenberg group $\bbbh^n$ 
satisfying a variety of metric conditions. Each of these conditions implies that such a mapping $f$ is 
tangent \emph{in some sense} to the horizontal distribution, i.e., that in some sense $f$ satisfies 
the \emph{contact equation}
$$
f^*\alpha = 0.
$$
We use a new and geometric argument to prove that the contact equation greatly restricts the mapping's behavior. 
To illustrate this method, our main focus in this paper is a new proof of the pure $k$-unrectifiability of 
the Heisenberg group $\bbbh^n$, when $k>n$: 
\begin{theorem}
\label{pure k unrect}  
Let $k>n$ be positive integers. Let $E \subset \bbbr^k$ be a measurable set, and let $f \colon E \to \bbbh^n$ 
be a Lipschitz mapping. Then $\H^k_{\bbbh^n}(f(E))=0$.
\end{theorem}
Here $\H^k_{\bbbh^n}$ denotes the Hausdorff measure with respect to the Carnot-Carath\'eodory metric in the
Heisenberg group.
This powerful result was proved by Ambrosio and Kirchheim~\cite[Theorem~7.2]{ambrosiok} in the case $n=1$. They
derived it by combining their  general results on metric differentiation, the area formula for mappings into metric spaces, and 
the Pansu differentiability of Lipschitz mappings from a subset of $\bbbr^k$ into $\bbbh^n$. 
In the case that $k=2n+2$, the result was proved earlier by David and Semmes~\cite[Section~11.5]{DS}.
In \cite{magnaniHous} and \cite{magnani}, Magnani expanded on the ideas of ~\cite{Pansu} and ~\cite{ambrosiok} to prove a very general theorem about Lipschitz mappings 
between Carnot groups that includes Theorem~\ref{pure k unrect} as a special case. 
These proofs are rather involved, and do not show in a straightforward way how basic geometric properties
of the Heisenberg group are responsible for the validity of Theorem~\ref{pure k unrect}. Our proof does not use any of the machinery employed by Ambrosio and Kirchheim. 

Instead, we employ the \emph{approximate derivative of $f$}, denoted at a point $x \in E$ by $\ap df_x$. Denote by $\H^k$ 
the Hausdorff measure in $\bbbr^k$, which coincides with the Lebesgue measure. To prove 
Theorem~\ref{pure k unrect}, we first show that a mapping $f$ as in the statement of the theorem satisfies a weak contact equation in 
the following sense: at $\H^k$-almost every $x \in E$, the image of the approximate derivative of $f$ at $x$ is contained in the horizontal distribution at $f(x)$, i.e.,
\begin{equation}
\label{approx weak contact} 
\im \ap df_x \subset \ker \alpha(f(x)) \quad \text{at $\H^k$-almost every $x \in E$}.
\end{equation}
See Lemma \ref{3.2} below. 

The proof continues by exploiting the geometric implications of the weak contact equation \eqref{approx weak contact}, and shows directly how the geometry of the Heisenberg group affects the behavior of Lipschitz mappings from Euclidean spaces. As an important step in this process, we prove:
\begin{theorem}
\label{low rank lipschitz} 
Let $k>n$ and let $E \subset \bbbr^{k}$ be a measurable set. 
If $f \colon E \to \bbbh^n$ is locally Lipschitz, then for $\H^k$-almost every point $x \in E$,
$$
\rank \ap df_x \leq n.
$$
\end{theorem}
\begin{remark}
\label{rem}
We will actually prove a stronger result:
{\em the image of $\ap d(\pi\circ f)_x$ is an isotropic subspace of
$\bbbr^{2n}$ for almost all $x\in E$, where
$\pi:\bbbr^{2n+1}\to\bbbr^{2n}$ is the orthogonal projection defined in \eqref{proj}.}
\end{remark}
In the case when $E$ is open, the proof of Theorem \ref{low rank lipschitz} is much easier and well 
known, see e.g.\ \cite[Proposition~1.1]{hajlaszst}, \cite{magnani2}.
Namely, the mapping $f$ is locally Lipschitz continuous as a mapping into $\bbbr^{2n+1}$
and it satisfies the contact equation almost everywhere in the standard sense of differential forms. 
Taking the exterior derivative of this equation in the distributional sense
then easily leads to the result. 

When $E$ is only assumed to be measurable, the mapping is still locally
Lipschitz continuous as a mapping from $E$ to $\bbbr^{2n+1}$, and can be extended to a Lipschitz mapping 
from $\bbbr^k$ to $\bbbr^{2n+1}$. However, the validity of the contact equation can be guaranteed only on the set $E$, 
preventing the straightforward use of distributional derivatives. We overcome this substantial difficulty by 
replacing analytic methods with geometric arguments. 

Finally Theorem~\ref{pure k unrect} is deduced from Theorem~\ref{low rank lipschitz} by arguments related to the
proof of the Sard theorem and adapted to the metric structure of the Heisenberg group.
Connection to the Sard theorem should not be surprising: Theorem~\ref{low rank lipschitz}
implies that almost all points are critical with a strong estimate for the rank of the derivative.

The methods demonstrated in our proof of Theorem~\ref{low rank lipschitz} have other applications as well. 
In particular, they work well in the Sobolev category, and allow us to prove the following Theorem:

\begin{theorem}
\label{low rank} 
Let $k>n$, and let $\Omega$ be an open subset of $\bbbr^k$. 
If $f \in W^{1,1}_{\loc}(\Omega;\bbbr^{2n+1})$ satisfies the weak contact equation
\begin{equation} 
\label{Sob weak contact} 
\im \wk df_x \subset \ker \alpha(f(x)) \quad \text{for $\H^k$-almost every $x \in \Omega$},
\end{equation}
then for 
$\H^k$-almost every $x \in \Omega$, 
$$
\rank \wk df_x \leq n.
$$
\end{theorem} 

Here $W^{1,1}_{\loc}(\Omega;\bbbr^{2n+1})$ denotes the Sobolev space of locally integrable mappings 
$f \colon \Omega \to \bbbr^{2n+1}$ with a locally integrable weak derivative $\wk df$.  Magnani, \cite{magnani2}, proved Theorem~\ref{low rank} in the case $n=1$, $k=2$, and 
$f\in W^{1,p}(\Omega; \bbbr^{3})$ for some $p\geq 4/3$. His argument based on the notion of distributional Jacobian cannot be 
applied in the case $1\leq p<4/3$. This range was left as an open problem. Recently and independently, Theorem~\ref{low rank} was also proved by 
Magnani, Mal\'y, and Mongodi in~\cite{magnanimm}, using analytic methods. 

Remark~\ref{rem} applies to Theorem~\ref{low rank} as well. Moreover, the assumption of Sobolev regularity 
cannot be reduced to bounded variation, as shown in \cite{balogh2}.

We will also present yet another proof of Theorem~\ref{low rank}. 
Since Theorem~\ref{low rank} is local in nature, 
we may assume that $\Omega$ is an open ball, and that $f$ and $\wk df$ are integrable on this ball.
Thus, Theorem~\ref{low rank} follows directly from
Theorem~\ref{low rank lipschitz} and the following result.
\begin{theorem}
\label{bd_lip}
Let $\Omega$ be a bounded domain in $\bbbr^k$ with smooth boundary.
If $f \in W^{1,1}(\Omega;\bbbr^{2n+1})$ satisfies the weak contact equation \eqref{Sob weak contact},
then for any $\eps>0$ there is is a set $E_\eps$ such that $\H^k(\Omega\setminus E_\eps)<\eps$ and
$f|_{E_\eps}:E_\eps\to\bbbh^n$ is Lipschitz continuous.
\end{theorem}
Here we do {\em not} require $k>n$.  
The proof of this result is based on the methods of analysis on metric spaces and is very different in 
nature from other proofs presented in this paper. 
Combined with Theorem \ref{pure k unrect}, it also has the following corollary:

\begin{corollary}
\label{small image Sobolev} 
Assume the notation and hypotheses of Theorem \ref{low rank}. If in addition, the mapping $f$ satisfies the Lusin condition
\begin{equation}
\label{Nkk} 
\text{if $Z \subseteq \Omega$ and $\H^k_{\bbbr^k}(Z)=0$, then} \ \H^k_{\bbbh^n}(f(Z))=0,
\end{equation}
then $\H^k_{\bbbh^n}(f(\Omega))=0$. 
\end{corollary}

On the other hand, the following can be deduced from \cite{hajlaszt} and \cite[Proposition~6.8]{DHLT}:
\begin{example}\label{critical space fill}
For each integer $k\geq 2$,  there is a bounded and continuous mapping $f \in W^{1,k}(B_{\bbbr^k}(0,1);\bbbr^{2n+1})$ that satisfies the 
weak contact equation \eqref{Sob weak contact} and whose image contains an open set in $\bbbh^n$.
\end{example} 

Corollary \ref{small image Sobolev} implies that when $k \leq 2n+2$, a mapping as in Example \ref{critical space fill} cannot satisfy the Lusin condition \eqref{Nkk}. However, in the presence of even slightly better Sobolev regularity, the condition \eqref{Nkk} is guaranteed; see, e.g., \cite{EucPeano}. 
This implies the following result, which can be considered as a Sobolev version of Theorem~\ref{pure k unrect}:

\begin{corollary}
\label{super critical} 
Let $k>n$, and let $\Omega$ be an open 
subset of $\bbbr^k$.  If $p>k$ and $f \in W^{1,p}_{\loc}(\Omega;\bbbr^{2n+1})$ 
satisfies the weak contact equation \eqref{Sob weak contact}, then 
$$
\H^k_{\bbbh^n}(f(\Omega))=0.
$$ 
\end{corollary}

In Corollary \ref{super critical}, it is important that the weak contact condition hold almost everywhere in the \emph{open} set $\Omega$, as the following example from \cite{Crelles} shows:

\begin{example}\label{cantor} There is a continuously differentiable mapping $f \colon \bbbr^2 \to \bbbr^3$ and a set $E \subeq \bbbr^2$ such that 
$$\im \ df_x \subeq \ker \alpha(f(x)) \ \text{for all $x \in E$} $$
and $\H^2_{\bbbh^1}(f(E))>0$.
\end{example} 

As a final application of our method, we prove a result related to the following theorem of Gromov, proven in \cite[Corollary~3.1.A]{gromov}.
\begin{theorem}
\label{Gromov theorem} Let $0<\alpha \leq 1$, and let $k>n$. 
Every $\alpha$-H\"older continuous embedding $f:\bbbr^k\to\bbbh^n$ satisfies $\alpha\leq\frac{n+1}{n+2}$.
\end{theorem} 
A well-known conjecture attributed to Gromov states that in the above theorem, it is actually the case that $\alpha\leq 1/2$. A number of tools for attacking this problem have recently been introduced to the literature, such as \cite{Zust} and \cite{EnZust}.

Using Theorem~\ref{pure k unrect}, we confirm this conjecture for mappings that are 
additionally assumed to be locally Lipschitz with respect to the Euclidean metric in $\bbbr^{2n+1}$, as we now describe.

Let $0<\alpha<1$. We say that
a mapping $f \colon X \to Y$ between metric spaces is in the class $C^{0,\alpha+}(X;Y)$ if there is a non-decreasing continuous function $\beta \colon [0,\infty) \to [0,\infty)$ satisfying $\beta(0)=0$ such that for all $x,y \in X$,
$$
d_Y(f(x),f(y)) \leq d_X(x,y)^{\alpha}\beta(d_X(x,y)).
$$

\begin{theorem}
\label{Lip Gromov} 
Let $k>n$ be positive integers, and let $\Omega$ be an open subset of $\bbbr^k$. Then there is no injective mapping in the class $C^{0,\frac{1}{2}+}(\Omega;\bbbh^n)$ that is also locally Lipschitz when considered as a map into $\bbbr^{2n+1}$.
\end{theorem}
This result is also related to an example constructed in~\cite{hajlaszm}.
A brief sketch of the proof of Theorem~\ref{Lip Gromov} is as follows; we postpone the 
full proof to Section~\ref{gromov conjecture}.  
Theorem~\ref{Lip Gromov} follows from Theorem~\ref{pure k unrect}, Rademacher's theorem, and two 
elementary facts. Rademacher's theorem states that a locally Lipschitz mapping between open subsets 
of Euclidean spaces is classically differentiable almost everywhere. The first elementary fact needed, 
Proposition~\ref{contact} below, implies that if a map $f \in C^{0,(1/2)+}(\Omega;\bbbh^n)$ is also 
locally Lipschitz as a map into $\bbbr^{2n+1}$, then for $\H^k$-almost every point $x \in \Omega$, the image of the classical derivative $df_x$ is contained in the horizontal tangent space ${\rm H}_{f(x)}\bbbh^n$, i.e., 
$f$ satisfies the classical contact equation almost everywhere. The second fact, Proposition~\ref{get Heis Lip} 
below, shows that this property implies that $f \colon \Omega \to \bbbh^n$ is locally Lipschitz. 
Theorem~\ref{pure k unrect} now yields that $\H^k_{\bbbh^n}(f(\Omega))=0$, from which it follows that $f$ cannot be an embedding.

The paper is organized as follows. In Section~\ref{AD} we provide basic definitions and results about approximate differentiability.
Section~\ref{heisenberg} is an introduction to the geometry of the Heisenberg group and contains all the basic definitions and results
that are used in the sequel. We also prove here that Lipschitz mappings into the Heisenberg group satisfy the 
weak contact equation \eqref{approx weak contact}. 
In Section~\ref{SF} we study elementary symplectic linear algebra as it relates to the Heisenberg group, and construct isometries of the Heisenberg group generated by Euclidean isometries between isotropic subspaces.
Section~\ref{RD} contains the proofs of Theorems~\ref{low rank lipschitz} and~\ref{low rank}, which are similar in nature. 
In Section~\ref{ih}, we prove Theorem~\ref{bd_lip}, giving a second proof of Theorem~\ref{low rank}.
We also prove here Corollaries~\ref{small image Sobolev} and \ref{super critical}.
Section~\ref{main} contains the proof of Theorem~\ref{pure k unrect}.
The arguments used here are closely related to the proof of Sard's theorem. The final Section~\ref{gromov conjecture}
contains the proof of Theorem~\ref{Lip Gromov}. 

Our notation throughout is fairly standard. By $C$ will denote a positive constant whose value may change in a single string of estimates.
The integral average will be denoted by
$$
\barint_E f\, d\mu = \mu(E)^{-1}\int_E f\, d\mu.
$$
The $\alpha$-dimensional Hausdorff measure in $\bbbr^k$ and $\bbbh^n$ will be denoted by
$\H^\alpha$ and $\H^\alpha_{\bbbh^n}$ respectively.
However, in order to avoid confusion, the Hausdorff measure in $\bbbr^{2n+1}$ with respect to the 
Euclidean metric will occasionally also be denoted by $\H^\alpha_{\bbbr^{2n+1}}$ to clearly distinguish it from
$\H^\alpha_{\bbbh^n}$.
Other notation will be explained as it arises.

\section{Approximate differentiability}
\label{AD}

Let $E \subset \bbbr^k$ be a measurable set. We say that a function $f \colon E \to \bbbr$ is
\emph{approximately differentiable at almost every point of $E$} if for every $\eps>0$, 
there is a set $K \subset E$ such that $\H^k(E\setminus K) < \eps$ 
and $g\in C^1(\bbbr^k)$ that agrees with $f$ on $K$. At any $\H^k$-density point $x$ of the set $K$, we define the 
\emph{approximate derivative of $f$ at $x$} by
$$
\ap df_x = dg_x.
$$
Note that $\ap df_x$ is independent of the choice of the function $g$ and set $K$  used to define it. 
A mapping $f:E\to \bbbr^m$ is approximately differentiable a.e.\ in $E$ if its coordinate
functions are approximately differentiable a.e.\ in $E$.

It is a theorem of Whitney~\cite{whitney} that this definition of approximate differentiability at almost every point coincides with the almost-everywhere validity of the usual pointwise definition of approximate differentiability given, 
for example, in \cite[Section 6]{EG}. In particular if $f$ is differentiable a.e.\ in an open set $\Omega\subset\bbbr^k$, 
then it is approximately differentiable a.e.\ in $\Omega$ in the sense described above. 
If $f:E\to\bbbr$, $E\subset\bbbr^k$ is Lipshitz continuous, then it can be extended to
a Lipschitz function $F:\bbbr^k\to\bbbr$. This can be done using the McShane extension \cite[2.10.44]{federer}. 
According to Rademacher's theorem, $F$ is differentiable a.e., which in turn implies
that $f$ is approximately differentiable a.e.\ in $E$.

A Sobolev function $f \in W^{1,1}_{\loc}(\Omega)$ is approximately differentiable a.e.\ in $\Omega$, and 
at almost every such point $x$, it holds that
$$
\ap df_x = \wk df_x,
$$
where $\wk df_x$ denotes the weak derivative of $f$ \cite[Section~6]{EG}. Indeed, Sobolev functions
$f\in W^{1,1}(\bbbr^k)$ satisfy the pointwise inequality 
$$
|f(x)-f(y)|\leq C(k)|x-y|(\M|\nabla f|(x)+\M|\nabla f(y)|)
\quad
\mbox{a.e.}
$$
where $\M g$ denotes the Hardy-Littlewood maximal function. Hence $f$ is Lipschitz continuous on the set
$\{x:\, \M|\nabla f|(x)\leq t\}$, the complement of which has measure bounded above by a constant multiple of $t\inv$.  
Now the result follows from the fact that Lipschitz functions are approximately differentiable a.e.\ For more details 
see \cite{hajlasz2}.

In the next section we will discuss the approximate differentiability of mappings into the Heisenberg group.

\section{The Heisenberg group}
\label{heisenberg}

For more details and references to the results stated here without proof,
see, e.g., \cite{capogna}. We retain the notation of the introduction.

As mentioned in the introduction, the {\em Heisenberg group}  is a Lie group
$\bbbh^n=\bbbc^n\times\bbbr=\bbbr^{2n+1}$ equipped with the group law
$$
(z,t)*(z',t')=\left(z+z',t+t'+2\, {\rm Im}\,  \left(\sum_{j=1}^n z_j
  \overline{z_j'}\right)\right).
$$
We recall that the basis of left invariant vector fields $X_1,\hdots, X_n,Y_1,\hdots, Y_n, T$ and the horizontal distribution $H\bbbh^n$ were defined in the introduction. 
An absolutely continuous curve 
$$
\gamma=(\gamma^{x_1},\gamma^{y_1},\ldots,\gamma^{x_n},\gamma^{y_n},\gamma^t):[a,b]\to\bbbr^{2n+1}
$$ 
is called {\em horizontal} if $\gamma'(s)\in H_{\gamma(s)}\bbbh^n$ for almost every $s$. 
This condition is equivalent to the contact equation $\gamma^*\alpha=0$, and hence it is also equivalent to
$$
2\sum_{j=1}^n \gamma^{x_j}d\gamma^{y_j}-\gamma^{y_j}d\gamma^{x_i} = - d\gamma^t.
$$
Thus if the curve $\gamma $ is closed,
\begin{equation}
\label{seven}
\frac{1}{2} \sum_{j=1}^n 
\int_\gamma\gamma^{x_j}d\gamma^{y_j}-\gamma^{y_j}d\gamma^{x_i} =
-\frac{1}{4}\int_\gamma d\gamma^t = 0.
\end{equation}
If $n=1$, according to Stokes' theorem
the integral on the left hand side equals the oriented area enclosed
by the projection of $\gamma$ on the $x_1y_1$ plane, and hence the enclosed area equals zero.
This property will play a crucial role in our arguments.
Since the enclosed area equals zero, the
curve cannot be Jordan. Typically it looks like the figure $8$, perhaps with a larger number of loops.
If $n>1$ the sum on the left hand side equals the sum of oriented areas of projections on the
planes $x_jy_j$. This makes our arguments slightly more complicated, but
they are essentially the same as in the case $n=1$.

The distribution $H\bbbh^n$ is equipped with the left invariant sub-Riemannian metric $\boldg$
defined by the condition that the vectors $X_1(p),\ldots,X_n(p),Y_1(p),\ldots,Y_n(p)$ are
orthonormal at every point $p\in\bbbh^n$. 
The Heisenberg group $\bbbh^n$ is then equipped with the {\em Carnot-Carath\'eodory metric} $d_{cc}$ 
which is defined as the infimum of the lengths of
horizontal curves connecting two given points. The length $\ell_{\bbbh}(\gamma)$ of the curve is
computed with respect to the metric $\boldg$ on $H\bbbh^n$.
That is if
$$
\gamma'(t)=\sum_{i=1}^n \alpha_i(t)X_i(\gamma(t))+\beta_i(t)Y_i(\gamma(t)),
\quad
a\leq t\leq b,
$$
then
$$
\ell_{\bbbh}(\gamma) =
\int_a^b \Vert\gamma'(t)\Vert_{\bbbh}\, dt =
\int_a^b \sqrt{\sum_{i=1}^n \alpha_i^2(t)+\beta_i^2(t)}\, dt.
$$
Since
$$
\gamma'(t)=\sum_{i=1}^n \alpha_i(t)\frac{\partial}{\partial x_i} + 
\beta_i(t)\frac{\partial}{\partial y_i} +
\left(\sum_{i=1}^n 2\gamma^{y_i}(t)\alpha_i(t) -2\gamma^{x_i}(t)\beta_i(t)\right)\, \frac{\partial}{\partial t}
$$
we conclude that the Carnot-Carath\'eodory length $\ell_{\bbbh}(\gamma)$ is less than or equal to
the Euclidean length of $\gamma$. In fact, $\ell_{\bbbh}(\gamma)$ equals to the Euclidean length
of the projection of $\gamma$ on $\bbbr^{2n}$ defined in \eqref{proj}.

The non-integrability of the horizontal distribution implies that any two points in $\bbbh^n$ can be connected by a
horizontal curve and hence $d_{cc}$ is a true metric. 
In fact, any pair of points can be connected by curve whose length equals the distance between the points. Such a curve is called a {\em geodesic}; note that there may be more than one geodesic connecting a given pair of points. We say that a metric space $(X,d)$ is a {\em geodesic space} if any pair of points can be connected by a geodesic. The Heisenberg group is an example of a geodesic space.

The Carnot-Carath\'eodory metric is
topologically equivalent to the Euclidean metric. Moreover, for any compact
set $K\subset \bbbh^n$ there is a constant $C\geq 1$ such that
\begin{equation}
\label{SReq1}
C^{-1}|p-q|\leq d_{cc}(p,q)\leq C|p-q|^{1/2}
\end{equation}
for all $p,q\in K$. The space $\bbbh^n$ is complete with respect to the metric $d_{cc}$.
In what follows, $\bbbh^n$ will always be regarded as the metric space
$(\bbbh^n,d_{cc})$. 
In particular, the identity mapping ${\rm id}\, :\bbbh^n\to\bbbr^{2n+1}$ is
locally Lipschitz continuous. Hence a Lipschitz mapping
$f:E\to\bbbh^n$, $E\subset\bbbr^k$, is locally Lipschitz as a mapping into 
$\bbbr^{2n+1}$. As discussed previously, this implies that it is approximately differentiable a.e.\ in $E$.

It is often more convenient to work the {\em Kor\'anyi metric}, which is
bi-Lipschitz equivalent to the Carnot-Carath\'eodory metric but much easier to compute.
The Kor\'anyi metric is defined by
$$
d_{\rm{K}}(p,q)=\Vert q^{-1}*p\Vert_{\rm{K}},
\quad
\mbox{where}
\quad
\Vert (z,t)\Vert_{\rm{K}}=\left(|z|^4+t^2\right)^{1/4}\, .
$$
A straightforward computation shows that for
$p=(z,t)=(x,y,t)$ and $q=(z',t')=(x',y',t')$ we have
\begin{eqnarray}
\label{kor}
d_{\rm{K}}(p,q)
& = &
\left(|z-z'|^4 +
\left|t-t'+2\sum_{j=1}^n(x_j'y_j-x_jy_j')\right|^{2}\right)^{1/4} \\
& \approx &
|z-z'| +
\left|t-t'+2\sum_{j=1}^n(x_j'y_j-x_jy_j')\right|^{1/2}\, . \nonumber
\end{eqnarray}
Here $f\approx g$ means that $C^{-1}f\leq g\leq Cf$ for some constant $C\geq 1$.

A curve in a metric space is called {\em rectifiable} if it has
finite length. Given a compact interval $I \in \bbbr$, an absolutely continuous path $\gamma \colon I \to \bbbr^{2n+1}$ has finite length by the fundamental theorem of calculus. Since a horizontal path is assumed to be Euclidean absolutely continuous, and its length calculated with respect to sub-Riemannian metric is no greater than its Euclidean length, each horizontal path is also rectifiable in $\bbbh^n$.

Every rectifiable curve in a metric space
admits an arc-length parameterization
\cite[Theorem~3.2]{hajlasz1}. With this parameterization the curve is
$1$-Lipschitz. If $\gamma:[a,b]\to\bbbh^n$ is Lipschitz, then it is also Lipschitz
as a curve in $\bbbr^{2n+1}$ and hence it is differentiable a.e.
It turns out that the tangent vectors to $\gamma$ are horizontal a.e.,
so the curve is horizontal (see Proposition~\ref{contact} and also \cite[Proposition~11.4]{hajlaszk}
for a more general result). 
Thus any rectifiable curve in $\bbbh^n$ admits an arc-length parameterization in which it is horizontal.

We will need the following extension results for geodesic spaces. 
Clearly they apply to the Heisenberg group $X=\bbbh^n$ which is complete
as a metric space.
\begin{lemma}
\label{line extension Lipschitz}
Let $K\subset [a,b]$, be a compact subset of a compact interval. 
If $f:K\to X$ is an $L$-Lipschitz mapping into a geodesic space, then there is
an $L$-Lipschitz extension $F:[a,b]\to X$ which agrees with $f$ on $K$.
If in addition $X$ is complete, $K$ can be any subset of $[a,b]$,
not necessarily compact.
\end{lemma}
{\em Proof.}
Suppose first that $(X,d)$ is any geodesic metric space and $K$ is compact.
Let $\alpha=\inf K$ and $\beta =\sup K$ 
be the ``endpoints'' of the set $K$. By extending $f$ to the intervals $[a,\alpha]$
and $[\beta,b]$ as a constant mapping equal $f(\alpha)$ and $f(\beta)$ respectively
we can assume without loss of generality that $a,b\in K$.
Now we can write $[a,b]\setminus K$ as a countable union of disjoint open intervals
$\{ (a_i,b_i)\}_{i=1}^\infty$. Let
$\gamma_i:[0,d(f(a_i),f(b_i))]\to X$
be a geodesic connecting $f(a_i)$ to $f(b_i)$, parameterized by arc-length. Since
$\gamma_i$ is $1$-Lipschitz, the curve
$$
F_i:[a_i,b_i]\to X,
\qquad
F_i(t)=\gamma_i\left((t-a_i)\frac{d(f(a_i),f(b_i))}{|a_i-b_i|}\right)
$$
is Lipschitz with the Lipschitz constant bounded by
$$
\frac{d(f(a_i),f(b_i))}{|a_i-b_i|}\leq L.
$$
It now follows from the triangle inequality that
$F:[a,b]\to X$ defined by
$$
F(t) = \begin{cases}
          f(t) \quad \mbox{for $t\in K$},\\
          F_i(t) \quad \mbox{for $t\in (a_i,b_i)$},
         \end{cases}
$$
is an $L$-Lipschitz extension of $f$.
Suppose now that in addition to being geodesic, the space $X$ is also complete. If $K\subset [a,b]$
is any subset, then $f$ uniquely extends to the closure of $K$
as an $L$-Lipschitz mapping and the result follows from the compact $K$ case discussed above.
\hfill $\Box$

The next result is a variant of Lemma~\ref{line extension Lipschitz}
and will be used in the proof of Theorem~\ref{low rank lipschitz}. 

\begin{lemma}
\label{Jordan extension} 
Let $K\subset S$ be a compact subset of a circle $S\subset\bbbr^2$. 
We assume that the circle is equipped with the metric inherited from $\bbbr^2$.
If $f:K\to X$ is an $L$-Lipschitz mapping into a geodesic space, then there is an $L\pi/2$-Lipschitz
extension $F:S\to X$ which agrees with $f$ on $K$.
If in addition $X$ is complete, $K$ can be any subset of $S$,
not necessarily compact.
\end{lemma}
{\em Proof.}
If $K\subset S$ is compact,
we can write $S\setminus K$ as a countable union of arcs
$\{ \widetilde{a_ib_i}\}_{i=1}^\infty$. It is easy to see  
that there is a $1$-Lipschitz mapping from 
$\widetilde{a_ib_i}$ onto the interval of length $|a_i-b_i|$
$$
h_i:\widetilde{a_ib_i}\to [0,|a_i-b_i|].
$$ 
Let $\gamma_i$ be defined as in the proof of Lemma~\ref{line extension Lipschitz}. Now
$$
F_i:\widetilde{a_ib_i}\to X
\qquad
F_i(t)=\gamma_i\left(\frac{d(f(a_i),f(b_i))}{|a_i-b_i|}\, h_i(t)\right)
$$
is Lipschitz continuous with the Lipschitz constant bounded by $L$.
We define the extension $F$ as in Lemma~\ref{line extension Lipschitz}.
The mapping $F$ maps arcs of length $\ell$ onto curves of length
at most $L\ell$. Since the length of a shorter arc connecting two 
given points $a,b\in S$ is bounded by $|a-b|\pi/2$ the result easily follows.
If $X$ is also complete and $K\subset S$ is any subset, the argument is exactly the same as in the previous proof.
\hfill $\Box$

As an application of Lemma~\ref{line extension Lipschitz} we will prove

\begin{lemma}
\label{3.2}
Let $k$ and $n$ be positive integers. Let $E\subset\bbbr^k$ be measurable. If $f:E\to \bbbh^n$ is 
locally Lipschitz,
then for $\H^k$-almost every $x\in E$, the image of 
$\ap df_x$ is contained in the horizontal subspace $H_{f(x)}\bbbh^n$.
\end{lemma}
{\em Proof.}
Since measurable sets can be exhausted (up to a subset of measure zero)
by countably many compact sets, we may assume
that $E$ is compact and $f$ is $L$-Lipschitz, $L \geq 1$.  We may also assume without loss of generality that E is contained in the unit cube $[0,1]^k$.
For each $\rho \in [0,1]^{k-1}$, define the line segment $\ell_\rho = \{\rho\}\times[0,1]$, and set 
$$
E_{k-1} = \{\rho \in [0,1]^{k-1} : \ell_\rho \cap E \neq \emptyset\}.
$$ 
By Lemma~\ref{line extension Lipschitz}, for each
$\rho \in E_{k-1}$, we may find an $L$-Lipschitz extension 
$f_\rho$ of $f|_{\ell_\rho \cap E}$ to all of $\ell_\rho$.  
Since $f_\rho$ defines a rectifiable curve in $\bbbh^n$, it follows that for $\H^1$-almost every $s \in [0,1]$, 
the tangent vector $f'_\rho(s)$ exists and is contained in the horizontal subspace $H_{f_\rho(s)}\bbbh^n$.
 
Denote 
$$
\mathcal{B} =E \setminus \{x=(\rho,s) \in E: \ap df_x(e_k) \ \text{and}\ f'_{\rho}(s) \ \text{exist and agree}\}.
$$
We claim that $\H^k(\mathcal{B}) = 0$, showing that $\ap df_x(e_k)$ is in the horizontal 
tangent space for $\H^k$-almost every $x \in E$. An analogous argument for each vector 
$\ap df_x(e_j)$, $j=1,\hdots,k-1$, completes the proof. 

If $\H^k(\mathcal{B}) >0$, then we may find a set $K \subset \mathcal{B}$ such that $\H^k(K) > 0$ and a 
$C^1$-mapping $g \colon [0,1]^k \to \bbbr^{2n+1}$ that agrees with $f$ on $K$. Moreover,  
$dg_x(e_k)=\ap df_x(e_k)$ at each density point of $K$. 
Set 
$$
K_{k-1} = \{\rho \in [0,1]^{k-1} : \ell_\rho \cap K \neq \emptyset\},
$$
and given $\rho \in K_{k-1}$, set
$$
K_{\rho} = \{s \in [0,1]: (\rho,s) \in K\}.
$$
Fubini's theorem implies that for $\H^{k-1}$-almost every point $\rho \in K_{k-1}$ and 
$\H^1$-almost every $s \in K_\rho$, the point $x=(\rho,s)$ is a $\H^k$-density point of $K$, a 
$\H^1$-density point of $\ell_\rho \cap K$, and a point of differentiability of $f_\rho$. 
At such a point $x=(\rho,s)$, it holds that 
$$
\ap df_x(e_k)=dg_x(e_k)=f'_\rho(s),
$$
and so $x \notin \mathcal{B}$, a contradiction.
\hfill $\Box$

\section{The standard symplectic form}
\label{SF}

The {\em standard symplectic form} in $\bbbr^{2n}$ is the differential $2$-form on $\bbbr^{2n}$ defined by
$$
\omega=\sum_{i=1}^n dx_i\wedge dy_i,
$$
i.e. for vector fields
\begin{equation}
\label{vectors}
v=\sum_{i=1}^n v^{x_i}\frac{\partial}{\partial x_i} + v^{y_i} \frac{\partial}{\partial y_i}
\quad
\mbox{and}
\quad
w=\sum_{i=1}^n w^{x_i}\frac{\partial}{\partial x_i} + w^{y_i} \frac{\partial}{\partial y_i}
\end{equation}
we define $\omega(v,w) \colon \bbbr^{2n} \to \bbbr$ by
$$
\omega(v,w)=
\sum_{i=1}^n \left( v^{x_i}w^{y_i}- v^{y_i}w^{x_i}\right).
$$ 
We will denote the evaluation of this form at the point $q$ by $\omega(q)(v,w)$. When $v$ and $w$ are constant vector fields or are only defined at a single point, i.e., they are tangent vectors, $\omega(v,w)$ can be thought of as a single real number. 

The standard symplectic form can be equivalently defined by the \emph{standard complex structure on $T_q\bbbr^{2n}$}, i.e., the isomorphism $\cJ \colon T_q\bbbr^{2n} \to T_q\bbbr^{2n}$ determined by 
$$ 
\cJ\left(\frac{\partial}{\partial x_i} \right) = \frac{\partial}{\partial y_i}, 
\qquad 
\cJ\left(\frac{\partial}{\partial y_i} \right) = -\frac{\partial}{\partial x_i}  
$$
for $i=1,\hdots,n$.  Namely, it holds that 
$$
\omega(q)(v,w)=-\langle v(q),\cJ w(q)\rangle,
$$
where $\langle\cdot,\cdot\rangle$ stands for the standard scalar product on $\bbbr^{2n}$.
 
A vector subspace $V\subset \bbbr^{2n}$ is said to be {\em isotropic} if
$\omega(v,w)=0$ for all $v,w\in V$. Subspaces that are not isotropic have a geometric interpretation that will help in understanding the
main idea behind the proofs of Theorems~\ref{low rank lipschitz} and~\ref{low rank}.
If $V\subset\bbbr^{2n}$ is non-isotropic, then there are vectors $v,w\in V$ such that
$$
\omega(v,w)=\sum_{i=1}^n (dx_i\wedge dy_i) (v,w) \neq 0.
$$
Observe that $(dx_i\wedge dy_i)(v,w)$ is the oriented area of the projection of the parallelogram with sides $v$ and $w$ 
onto the coordinate plane $x_iy_i$. Thus the sum of the (oriented) areas of the projections on the planes
$x_iy_i$, $i=1,2,\ldots,n$, is different from zero. Clearly if we replace the parallelogram with any measurable set
$E\subset {\rm span}\, \{v,w\}$ of positive measure, the sum of oriented areas of the projections is still non-zero.
In the proofs of Theorems~\ref{low rank lipschitz} and~\ref{low rank} we will consider the case that $E$ is an ellipse.
 
If we identify elements $z$ and $z'$ of $\bbbr^{2n}=\bbbc^n$ with vectors of the form \eqref{vectors}, then the
product in the Heisenberg group can be written as
$$
(z,t)*(z',t')=(z+z',t+t'-2\omega(z,z')).
$$
Also
\begin{equation}
\label{dK_sym}
d_{\rm{K}}((z,t),(z',t'))=\left( |z-z'|^4+|t-t'-2\omega(z,z')|^2\right)^{1/4}.
\end{equation}
It easily 
follows from the definitions that for any isotropic subspace $V$ and any $t \in \bbbr$, 
the restriction of the Kor\'anyi metric to 
$V \times \{t\} \subset \bbbh^n$ agrees with the Euclidean metric on $V$. For this reason, we now 
discuss the well-known linear algebra associated with the standard symplectic form.

Given a subspace $V$ of $\bbbr^{2n}$, we denote by $V^*$ the \emph{dual} of $V$, i.e., the vector space of 
linear homomorphisms from $V$ to $\bbbr$, and define the \emph{symplectic complement} of $V$ by
$$
V^\omega = \{w \in \bbbr^{2n} : \omega(v,w) = 0 \ \text{for all $v \in V$}\}.
$$ 
Note that $V$ is isotropic if and only if $V \subset V^\omega$. An isotropic subspace $V$ is said to be \emph{Lagrangian} if it is of dimension $n$, the maximum possible for an isotropic subspace.

\begin{lemma}
\label{dim_n}
Let $V$ be a subspace of $\bbbr^{2n}$. Then 
\begin{itemize}
\item $\dim V^\omega = 2n - \dim V,$
\item if $V$ is isotropic, then $\dim V\leq n$,
\item if $V$ is isotropic, then there is a Lagrangian subspace of $\bbbr^{2n}$ that contains $V$.
\end{itemize}
\end{lemma}

\begin{proof}
Consider the homomorphism $\Phi \colon\bbbr^{2n} \to V^*$ defined by 
$$
\Phi(w)(v) = \omega(v,w) = -\langle v, \cJ w\rangle.
$$
The kernel of $\Phi$ is precisely $V^{\omega}$.  Since $\cJ$ is an isomorphism, the homomorphism $\Phi$ is surjective, showing that 
\begin{equation}\label{dim of Vomega} \dim V^{\omega} = 2n - \dim V^* = 2n - \dim V.\end{equation}

If $V$ is isotropic, then $V$ is contained in $V^\omega$, and so \eqref{dim of Vomega} implies that $\dim V \leq n.$

Finally, we  show that if $V$ is isotropic, then it is contained in a Lagrangian subspace. Let $V'$ be an isotropic subspace containing $V$ of maximal dimension. If $\dim V' = \dim (V')^\omega$, then \eqref{dim of Vomega} applied to $V'$ implies that $\dim V' = n$, showing that $V'$ is Lagrangian. If $\dim V' < \dim (V')^\omega$, then we may find a vector $w \in (V')^\omega$ that is not in $V'$. However, the subspace generated by $V' \cup \{w\}$ is again isotropic, contradicting the maximality of $V'$. \end{proof}

The canonical example of an isotropic subspace of $\bbbr^{2n}$ is the span of 
$\{\partial/\partial x_1,\hdots,\partial/\partial x_j\}$ for some $j \leq n$. 
From the perspective of the metric geometry of the Heisenberg group, all isotropic subspaces may be assumed to be of this form, as the following statement shows. A more detailed proof can be found in~\cite[Lemma~2.1]{balogh1}.

\begin{lemma}
\label{standard isotropic} 
Let $V$ and $W$ be isotropic subspaces of $\bbbr^{2n}$, which we identify with the 
corresponding subsets of $\bbbr^{2n} \times \{0\} \subset \bbbh^n$.  
If $V$ and $W$ have the same dimension, then there is a linear 
isometry $\Phi \colon \bbbh^n \to \bbbh^n$ with respect to the Kor\'anyi metric
such that $\Phi(V)=W$. 
\end{lemma} 
\begin{proof} Let $1\leq j \leq n$ denote the dimension of $V$ and $W$. By Lemma~\ref{dim_n}, 
we may find Lagrangian subspaces $V'$ and $W'$ so that $V$ is contained in $V'$ and $W$ is 
contained in $W'$. 
Since $V'$ is isotropic, the subspace $\cJ V'$ is orthogonal to
$V'$. Since $\dim V'=n$, the orthogonal sum of $V'$ and $\cJ V'$ is all of $\bbbr^{2n}$. The
same holds true for $W'$ and hence
$$
V'\oplus \cJ V'= W'\oplus \cJ W'=\bbbr^{2n}.
$$
Note that $\cJ$ is an orthogonal transformation of $\bbbr^{2n}$.
Using the Gram-Schmidt process, we may find orthonormal bases 
$$\mathcal{B}_V = \{v_1,\hdots,v_n\} \ \text{and}\ \mathcal{B}_W = \{w_1,\hdots w_n\}$$
of $V'$ and $W'$ respectively, such that 
\begin{itemize}
\item $\{v_1,\hdots, v_j\}$ and $\{w_1,\hdots,w_j\}$ are bases of $V$ and $W$, respectively,
\item $\mathcal{B}_V \cup \cJ\mathcal{B}_V $ and $\mathcal{B}_W \cup \cJ\mathcal{B}_W$ are orthonormal bases of 
$\bbbr^{2n}$,
\item $\mathcal{B}_V \cup \cJ\mathcal{B}_V $ and $\mathcal{B}_W \cup \cJ\mathcal{B}_W$ are symplectic bases of 
$\bbbr^{2n}$, i.e., 
$$
\omega(v_i,v_k) =  0 = \omega(\cJ v_i,\cJ v_k), \ \text{and}\ \omega(v_i,\cJ v_k) = 
\begin{cases} 1 & i=k,\\ 0 &  i \neq k.\\ \end{cases}
$$
$$
\omega(w_i,w_k) =  0 = \omega(\cJ w_i,\cJ w_k), \ \text{and}\ \omega(w_i,\cJ w_k) = 
\begin{cases} 1 & i=k,\\ 0 &  i \neq k.\\ \end{cases}
$$
\end{itemize}
We define $\phi \colon \bbbr^{2n} \to \bbbr^{2n}$ to be the (Euclidean) linear isometry defined by
\begin{align*} 
&\phi(v_i) = w_i, \  i=1,\hdots,n \\
					&\phi(\cJ v_i) = \cJ w_i \  i=1,\hdots, n, 
\end{align*}
and define $\Phi \colon \bbbh^n \to \bbbh^n$ by $\Phi(z,t) = (\phi(z),t).$
Since $\phi$ is linear, the fact that it maps a symplectic basis to a symplectic basis implies that it preserves the symplectic form.
Hence, considering points $(z,t)$ and $(z',t')$ in $\bbbh^n$ it follows from \eqref{dK_sym} that
\begin{equation}
\label{isometry} 
d_{\rm{K}}(\Phi(z,t),\Phi(z',t'))=d_{\rm{K}}((z,t),(z',t')),
\end{equation}
as desired. 
\end{proof}

\section{Rank of the derivative}
\label{RD}

The goal of this section is to prove Theorems~\ref{low rank lipschitz} and~\ref{low rank}.
Although the idea is geometric and elementary, the technical details of the proofs 
hide it like a needle in a haystack. Thus it is reasonable to spend a while trying to explain
the main idea before going into details.

We focus on Theorem~\ref{low rank lipschitz} first.
For simplicity suppose that $n=1$ and $k=2$. Let $E\subset \bbbr^2$ be a measurable set, and let $f \colon E \to \bbbh^1$ be a Lipschitz mapping. Suppose that on a set $K \subset E$ of positive $\H^2$-measure, the  rank of the
approximate derivative of $f$ is $2$. We may also assume that $f$ coincides with a $C^1$-mapping $g$ on $K$.
Choose a density point $x\in K$.
Since the approximate derivative $dg_x$ has rank $2$, it maps a small circle $S$ in the tangent space to $\bbbr^2$
onto an ellipse $dg_x(S)$. By Lemma \ref{3.2}, this ellipse lies in the horizontal plane at $f(x)$.
The projection of this ellipse on the $x_1y_1$ plane is also a non-degenerate ellipse (see \eqref{rank} below) 
and hence it bounds
a positive area.
As $g$ is continuously differentiable, the image $f(S\cap K)=g(S\cap K)$ is close to the ellipse $dg_x(S)$. Since $x$ is a density point of $K$, if the radius of the circle $S$ is sufficiently small, we may assume that $S\cap K$ accounts for most of the length of $S$. 
On the much shorter remaining set $S\setminus K$, the mapping $f$ is not necessarily defined.
However, we can extend $f$ from $S\cap K$ to all of $S$ as a Lipschitz curve $F \colon S \to \bbbh^1$, using Lemma~\ref{Jordan extension}.  The 
image of $F(S\setminus K)$ is also short in length. The resulting curve $F(S)$ is horizontal as it is Lipschitz.
Hence its projection onto the $x_1y_1$ plane bounds the oriented area zero.
However, the portion of the curve $F(S\setminus K)$ is short in length and $F(S \cap K)=f(S\cap K)$ is close to the ellipse
$dg_x(S)$, so the area of the projection of $F(S)$ does not differ
much from the area of the projection of $dg_x(S)$. Hence this area is positive, a contradiction.

This final step dealing with the area of the projection requires Stokes' theorem as described in Section~\ref{heisenberg}.
In the case $k>2$ we need to choose a suitable $2$-dimensional slice, and if $n>1$ we need to work with
areas of projections on all the planes $x_iy_i$, $i=1,\hdots,n$, at the same time. 
To do this, we use the symplectic form, which is the sum of volume forms in all the planes $x_iy_i$.

The idea of the proof of Theorem~\ref{low rank} is only slightly different.
The function $f$ is now defined on the whole domain $\Omega$ and not only on a measurable subset, so there is no
need to do an extension from a subset of a circle to the whole circle, making the argument more direct.
Again, for simplicity we suppose that $n=1$ and $k=2$. If the rank of the weak derivative $\wk df$ is $2$ at a point $x \in \bbbr^2$, then $\wk df_x$ maps a small circle $S$
onto the ellipse $\wk df_x(S)$ in the horizontal space, whose projection on the $x_1y_1$ plane bounds a region of non-zero oriented area.

However, the restriction of $f$ to generic small circles centered at $x$ is in the Sobolev space on the circle and hence absolutely continuous.
Thus it forms a horizontal curve, so its projection on the $x_1y_1$ plane bounds a region of zero oriented area.
On the other hand it follows from the Fubini theorem and standard tools from the theory of Sobolev spaces
that on generic small circles $S$ centered at $x$, the curve $f(S)$ is very close to the ellipse $df_x(S)$. Hence the
area bounded by the projection of $f|_S$ has to be close to the area bounded by the 
projection of the ellipse, again yielding a contradiction. 

As in the Lipschitz setting, when $k>2$ we need to choose a suitable $2$-dimensional slice and if $n>1$ we need to work with
areas of projections on all the planes $x_iy_i$ at the same time.

For the reminder of the section we fix positive integers $k>n$.
Points in $\bbbr^{2n+1}$ will be also denoted by $(x_1,y_1,\ldots,x_n,y_n,t),$ 
agreeing with the Heisenberg group notation. If $f$ is a mapping into $\bbbr^{2n+1}$, then we will write
$$
f=(f^{x_1},f^{y_1},\ldots,f^{x_n},f^{y_n},f^t).
$$
We denote by
$
\pi:\bbbr^{2n+1}\to\bbbr^{2n}
$
the orthogonal projection
\begin{equation}
\label{proj}
\pi(x_1,y_1,\ldots,x_n,y_n,t)=(x_1,y_1,\ldots,x_n,y_n).
\end{equation}
Vectors in $\bbbr^{2n}$ will often be denoted by
$$
v=(v^{x_1},v^{y_1},\ldots,v^{x_n},v^{y_n}) =
\sum_{i=1}^n v^{x_i}\frac{\partial}{\partial x_i} + v^{y_i} \frac{\partial}{\partial y_i}\, .
$$
For $p\in\bbbh^n$ we can regard $H_p\bbbh^n$ as a linear subspace of $T_p\bbbr^{2n+1}$. Then
\begin{equation}
\label{rank}
d\pi_p|_{H_p\bbbh^n}\colon H_p\bbbh^n\to T_{\pi(p)}\bbbr^{2n}
\quad
\mbox{is an isomorphism.}
\end{equation}
Indeed,
$$
d\pi_p(X_j(p))=\left.\frac{\partial}{\partial x_j}\right|_{\pi(p)},
\qquad
d\pi_p(Y_j(p))=\left.\frac{\partial}{\partial y_j}\right|_{\pi(p)},
$$
so the basis of $H_p\bbbh^n$ is mapped onto the canonical basis of
$T_{\pi(p)}\bbbr^{2n}$.

\subsection{Proof of Theorem~\ref{low rank lipschitz}}

By the previously stated exhaustion 
argument, we may assume that $E$ is compact and that $f$ is $L$-Lipschitz. Throughout
the proof constants $C$ will depend on $n$ and $\sup_{x\in E}\Vert f(x)\Vert_{\rm{K}}$ only.
The dependence on the last quantity will stem the fact that
the identity mapping from $f(E)\subset\bbbh^n$ to $\bbbr^{2n+1}$ is
Lipschitz continuous with the constant $C$ depending only on
$\sup_{x\in E}\Vert f(x)\Vert_{\rm{K}}$.

Since $f$ is approximately differentiable at almost every point of $E$, it suffices to prove that if
$g:\bbbr^{k}\to\bbbr^{2n+1}$ is of class $C^1$ and $K \subset E$ is such that $f|_K=g|_K$, then
\begin{equation}
\label{rank_dg}
\rank dg\leq n
\quad 
\mbox{$\H^k$-a.e. in $K$.}
\end{equation}
To this end it suffices to prove that the image of $d(\pi\circ g)_x$ is an isotropic
subspace of $T_{\pi(g(x))}\bbbr^{2n}$ for $\H^k$-almost every $x\in K$.
Indeed,~\eqref{rank_dg} will follow then from Lemma~\ref{dim_n},
Lemma~\ref{3.2}, and \eqref{rank} (in that order).

Suppose, by way of contradiction, that the set of points $x\in K$ such that the image of
$d(\pi\circ g)_x$ fails to be isotropic contains a set of positive measure.
In what follows we will identify $T_x\bbbr^k$ with $\bbbr^k$ through an obvious canonical
isomorphism. 
Let $\{e_1,\ldots,e_k\}$ be the canonical basis of $\bbbr^k$. If the image of 
$d(\pi\circ g)_x$ fails to be isotropic, then for some $i,j \in \{1,\hdots,k\}$, the pullback of $\omega$ by $\pi \circ g$ satisfies
\begin{equation}
\label{ij}
((\pi \circ g)^*\omega)(x)\left(e_i,e_j\right)\neq 0
\end{equation}
Thus for some fixed $i,j$, \eqref{ij}
holds for all points $x$ in a set $K$ of positive measure. Thus without loss of generality we may assume that
there is a $2$-dimensional subspace $V={\rm span}\,\{e_i,e_j\}$ of $\bbbr^k\simeq T_x\bbbr^k$
such that $d(\pi\circ g)_x(V)$ is non-isotropic for
every $x\in K$ and $\H^k(K)>0$. 

Let $v_1=e_i$ and $v_2=e_j$ be the given orthonormal basis of $V$. Our assumptions above mean that for every $x \in K$,
\begin{equation}
\label{non-isotropic formula} 
\sum_{i=1}^n dg^{x_i}_x \wedge dg^{y_i}_x (v_1, v_2) 
=
\left((\pi \circ g)^*{\omega}\right)(x)(v_1,v_2) \neq 0
\end{equation}
Fubini's theorem implies that there is a point $a\in \bbbr^k$ such that
$$
\H^2\left(K\cap (a+V)\right) >0
$$ 
 The restriction of 
$(\pi\circ g)^*\omega$ to $(a+V)$ defines a differential $2$-form 
on $(a+V)$ with continuous coefficients (because $g$ is of class $C^1$), 
i.e. there is a continuous function $c:(a+V)\to\bbbr$ such that
$$
(\pi\circ g|_{a+V})^*\omega(x)=c(x) dv_1\wedge dv_2.
$$
Clearly \eqref{non-isotropic formula} implies that $c(x)\neq 0$ for 
$x\in K \cap (a+V)$.
Continuity of the function $c$ yields that for every $x_0\in (a+V)$ we have
$$
\lim_{r \to 0} \frac{1}{\pi r^2} \int_{B(x_0,r) \cap (a+V)} (\pi\circ g|_{a+V})^*\omega= c(x_0). 
$$
Now Stokes' theorem implies that if $x_0\in K \cap (a+V)$, then for all sufficiently small $r>0$ we have
\begin{eqnarray}
\label{apply LDT} 
\lefteqn{0<\frac{|c(x_0)|\pi r^2}{2} \leq \left|  \int_{B(x_0,r) \cap (a+V)}(\pi\circ g|_{a+V})^*\omega \right|}\\ 
& = & 
\left| \frac{1}{2}\sum_{i=1}^n \int_{\partial{B}(x_0,r) \cap (a+V)}g^{x_i}dg^{y_i} -g^{y_i}dg^{x_i} \right|. \nonumber
\end{eqnarray}
Let $x_0 \in K \cap (a+V)$ be a density point of $K \cap (a+V)$. 
We could assume without loss of generality that $f(x_0)=g(x_0)=0$, because the left translation on the 
Heisenberg group is an isometry. This would slightly simplify notation -- we would not have to subtract $f(x_0)$
in the formulas that follow, but this would require the reader to check (or to believe) that indeed 
we are allowed to make this assumption. Instead we prefer direct computations without making 
this (clever) assumption.

Let $\eps>0$. By Fubini's Theorem, we may find a radius $r>0$ such that \eqref{apply LDT} holds and 
\begin{equation} 
\label{most of circle}
\H^1\left(\partial{B}(x_0,r) \cap ((a+V) \setminus K)\right) \leq \eps r.
\end{equation}
For ease of notation, we denote the disk $B(x_0,r) \cap (a+V)$ by $B_0$. 
Since $f$ is assumed to be  $L$-Lipschitz, by Lemma~\ref{Jordan extension} we may find an 
$L\pi/2$-Lipschitz mapping $F \colon \partial B_0  \to \bbbh^n$ such that 
$F|_{\partial B_0 \cap K}=f|_{\partial B_0 \cap K} = g|_{\partial B_0 \cap K}$. 
Since the identity map from $\bbbh^n$ to $\bbbr^{2n+1}$ is Lipschitz on compact sets, the mapping $F$ is also 
$CL$-Lipschitz when considered as a mapping into $\bbbr^{2n+1}$.  
Hence, for any $s \in \partial{B_0}$, the estimate \eqref{most of circle} implies that
\begin{equation}
\label{gaps} 
\dist_{\bbbr^{2n+1}}(F(s),f(\partial{B_0} \cap K)) \leq CL\eps r.
\end{equation}
Since $f$ is also $CL$-Lipschitz as a mapping into $\bbbr^{2n+1}$, \eqref{gaps} implies that for each 
$s \in \partial{B_0}$, the Euclidean distance between $F(s)$ and $f(x_0)$ is bounded by
$$
|F(s)-f(x_0)|\leq CLr .
$$ 
If we evaluate the integral below
in the arc-length parameterization $\gamma$ of $\partial B_0$, the terms $dF^{x_i}$ and
$dF^{y_i}$ will become $(F^{x_i}(\gamma(t))'$ and $(F^{y_i}(\gamma(t))'$ respectively and hence
they will be bounded by $CL$, because $F$ is $CL$-Lipschitz into $\bbbr^{2n+1}$. 
Thus \eqref{most of circle} implies that 
\begin{eqnarray}
\label{small int F}
\lefteqn{\left| \sum_{i=1}^n\int_{\partial{B_0}\setminus K} 
(F^{x_i}-f^{x_i}(x_0))dF^{y_i} - (F^{y_i}-f^{y_i}(x_0))dF^{x_i}\right|} \phantom{aaaaaaaaa}  \\ 
& \leq & 
CLr\cdot CL\cdot\H^1(\partial{B}_0 \setminus K) \leq C^2L^2r^2\eps. \nonumber
\end{eqnarray}
Since $F$, as a Lipschitz curve in $\bbbh^n$,  is horizontal, \eqref{seven} yields 
\begin{eqnarray}
\label{use horiz} 
\lefteqn{ \sum_{i=1}^n\int_{\partial{B}_0} 
(F^{x_i}-f^{x_i}(x_0))dF^{y_i} - (F^{y_i}-f^{x_i}(x_0))dF^{x_i}}\\ 
& = & 
\sum_{i=1}^n\int_{\partial{B}_0} 
F^{x_i}dF^{y_i} - F^{y_i}dF^{x_i} = 
-\frac{1}{2}\int_{\partial B_0} dF^t =0. \nonumber
\end{eqnarray}
As a result of \eqref{small int F}, \eqref{use horiz}, the fact that 
$F|_{\partial{B_0} \cap K}=g|_{\partial{B_0} \cap K}$, 
and $f(x_0)=g(x_0)$ we conclude that  
\begin{equation}
\label{y1}
\left| \sum_{i=1}^n\int_{\partial{B_0} \cap K} (g^{x_i}-g^{x_i}(x_0))dg^{y_i} - 
(g^{y_i}-g^{y_i}(x_0))dg^{x_i}\right| \leq C^2 L^2 r^2\eps.
\end{equation}
On the other hand, $g$ is also Lipschitz as a mapping into $\bbbr^{2n+1}$
with some constant $L'$ in a bounded region near $x_0$. Hence
the same reasoning that led to \eqref{small int F} also shows that 
\begin{equation}
\label{y2}
\left| \sum_{i=1}^n\int_{\partial{B_0} \setminus K} 
(g^{x_i}-g^{x_i}(x_0))dg^{y_i} - (g^{y_i}-g^{y_i}(x_0))dg^{x_i}\right| \leq C(L')^2 r^2\eps.
\end{equation}
After choosing $\eps$ small enough, the inequalities \eqref{y1} and \eqref{y2} lead to a contradiction with \eqref{apply LDT},
because in the last integral at \eqref{apply LDT} we can subtract
$g^{x_i}(x_0)$ from $g^{x_i}$ and $g^{y_i}(x_0)$ from $g^{y_i}$ without changing
its value. The proof of Theorem~\ref{low rank lipschitz} is complete.
\hfill $\Box$


\subsection{Proof of Theorem~\ref{low rank}}
Let us first briefly recall basic facts from the theory of Sobolev spaces that will be used in the proof.
For more details, see Chapters~4 and ~6 of \cite{EG}. If $u \in W^{1,1}_{\loc}(\bbbr^k,\bbbr^{2n})$, 
and $V$ is a linear subspace of $\bbbr^k$, then 
the restriction of $u$ to almost all subspaces parallel to $V$ behaves nicely. Namely,
for almost all $a\in\bbbr^k$ 
$$
u|_{a+V} \in W^{1,1}_{\loc}(a+V; \bbbr^{2n}).
$$
Moreover the weak derivative of the restriction of $u$ is the restriction of the  weak derivative of $u$. 
This follows easily from Fubini's theorem if $u$ is smooth, and in the general case it follows from 
Fubini's theorem applied to a smooth approximation of $u$.
Clearly the result can be generalized to the case of functions that are defined on an open subset
$\Omega\subset\bbbr^k$, in which case we consider restrictions to $\Omega\cap (a+V)$.
For $x_0\in\bbbr^k$, a similar result
holds for restrictions to spheres centered at $x_0$:
\begin{equation}
\label{res_sp}
u|_{S^{k-1}(x_0,r)} \in W^{1,1}(S^{k-1}(x_0,r);\bbbr^{2n})
\qquad
\mbox{for almost all $r>0$.}
\end{equation}
Sobolev mappings defined on one-dimensional manifolds are absolutely continuous. In particular, if $k=2$ and 
$\gamma:[a,b]\to\bbbr^2$ parameterizes an arc of a circle $S$, and the restriction of $u$ to $S$ is again a Sobolev mapping, then
\begin{equation}
\label{upper}
|u(\gamma(b))-u(\gamma(a))|\leq \int_\gamma |\nabla u|.
\end{equation}
To be precise, we also need to assume that the weak derivative of the restriction of $u$
to $S$ is the restriction of the weak derivative of $u$; this is the case for almost every circle centered at a given point. 

We will also need the following result of Calder\'on and Zygmund, which can be deduced from \cite[Theorem~6.1.2]{EG} and H\"older's inequality.
\begin{lemma}
\label{CZ}
If $u\in W^{1,1}(\bbbr^k;\bbbr^{2n})$, then 
$$
\lim_{r\to 0} \frac{1}{r}\, \barint_{B(x,r)} |u(y)-u(x)-du_x(y-x)|\, dy = 0
$$
for almost all $x\in\bbbr^k$.
\end{lemma}
Now we are ready to proceed to the proof of Theorem \ref{low rank}. By the usual exhaustion argument, we may assume that $f \in W^{1,1}(\Omega;\bbbr^{2n+1})$, and that it satisfies the weak contact equation
\begin{equation} 
\label{Sob weak contact 2} 
\im \wk df_x \subset \ker \alpha(f(x)) \quad \text{for $\H^k$-almost every $x \in \Omega$}.
\end{equation}
We wish to prove that for $\H^k$-almost every point $x \in \Omega$, the rank of $\wk df_x$ is no greater than $n$.
 
Set $u= \pi \circ f$, so that $u \in W^{1,1}(\Omega;\bbbr^{2n}).$ For ease of notation, we denote the weak derivative $\wk du$ simply as $du$. As before, it suffices to prove that the image of
$du_x$ is isotropic for almost all $x\in\Omega$.  Suppose that
this is not the case. Then again, following the arguments of the proof of 
Theorem~\ref{low rank lipschitz} we can find coordinate directions $e_i,e_j$ in
$\bbbr^k$ and $a\in\bbbr^k$ such that the image of $du$ along the two-dimensional slice $\Omega\cap (a+V)$,
where $V={\rm span}\, \{e_1,e_2\}$ is non-isotropic on a set 
$A\subset\Omega\cap (a+V)$ of positive measure.
This shows that we can assume that $k=2$, and 
that the image of 
$$
du_x:T_x\bbbr^2\to T_{\pi\circ f(x)}\bbbr^{2n}
$$
is non-isotropic for $x\in A\subset\Omega$, $\H^2(A)>0$.
In what follows $B(x,r)$ will denote a two-dimensional disc
and $S(x,r)$ its boundary, equipped with the usual length measure $\sigma$.

Just as discussed above, if $x\in\Omega$, then for almost all $0<r<\dist (x,\partial\Omega)$,
$f$ restricted to the circle $S(x,r)$ is absolutely continuous (as it is a Sobolev mapping). 
By Fubini's theorem, it also satisfies the contact condition for almost every $r$, and so for almost every $r$, the restriction $f|_{S(x,r)}$ defines a closed horizontal curve.

It follows from Lemma~\ref{CZ}
and the fact that almost all points are Lebesgue points of the mapping $x \mapsto du_x$
that 
$$
\lim_{r\to 0} \frac{1}{r}\, \barint_{B(x,r)} |u(y)-u(x)-du_x(y-x)|\, dy =0
$$
and
$$
\lim_{r\to 0} \barint_{B(x,r)}|du_y-du_x|\, dy = 0
$$
for almost all $x\in\Omega$. Fix such $x_0\in A\subset\Omega$. Let $\eps>0$.
Then there is $R>0$ (which depends on $\eps$) such that
\begin{equation}
\label{natalia}
\int_{B(x_0,R)} |u(y)-u(x_0)-du_{x_0}(y-x_0)|\, dy <\eps R^3,
\end{equation}
and
\begin{equation}
\label{joanna}
\int_{B(x_0,R)}|du_y-du_{x_0}|\, dy <\eps R^2.
\end{equation}
Since for any non-negative function $\alpha \colon B(x_0,R) \to [0,\infty]$,
$$
\int_{B(x_0,R)}\alpha\ dy\geq \int_{R/2}^R \int_{S(x_0,r)}\alpha\ d\sigma\, dr
$$
it follows from \eqref{natalia} and \eqref{joanna} that there are subsets
$K_1,K_2\subset [R/2,R]$, each of measure at least $3R/8$
such that
\begin{equation}
\label{e1}
\int_{S(x_0,r)} |u(y)-u(x_0)-du_{x_0}(y-x_0)|\, d\sigma(y) <8\eps R^2
\quad
\mbox{for $r\in K_1$}
\end{equation}
and
\begin{equation}
\label{e2}
\int_{S(x_0,r)} |du_y-du_{x_0}|\, d\sigma(y) < 8\eps R
\quad
\mbox{for $r\in K_2$.}
\end{equation}
The set $K_1\cap K_2\subset [R/2,R]$ has length greater than or equal to $R/4$
and for $r\in K_1\cap K_2$ both inequalities \eqref{e1} and \eqref{e2}
are satisfied. Note that $R\leq 2r$ for $r\in K_1\cap K_2$, so 
\begin{equation}
\label{e3}
\int_{S(x_0,r)} |u(y)-u(x_0)-du_{x_0}(y-x_0)|\, d\sigma(y) <32\eps r^2
\end{equation}
and
\begin{equation}
\label{e4}
\int_{S(x_0,r)} |du_y-du_{x_0}|\, d\sigma(y) < 16\eps r.
\end{equation}
The last inequality also gives
\begin{equation}
\label{e44}
\int_{S(x_0,r)} |du_y|\, d\sigma(y) <C r.
\end{equation}

Moreover, we may choose the radius $r\in K_1\cap K_2$ so that the function $f$ restricted to 
the circle $S(x_0,r)$ is in the appropriate Sobolev space and defines a closed horizontal curve.

The idea of the remaining part of the proof is as follows.
The fact that the image of
the derivative at $x_0\in A$ is non-isotropic means that 
$du_{x_0}$ maps circles to ellipses in $du_{x_0}(\bbbr^2)$ with the property that
the the sum of the oriented areas of projections on the planes
$x_iy_i$ is different than zero. Thus the ellipse parametrized by
\begin{equation}
\label{ellipse}
S(x_0,r)\ni y\mapsto u(x_0)+du_{x_0}(y-x_0)
\end{equation}
has this property. Conditions \eqref{e3} and \eqref{e4} mean that the curve
\begin{equation}
\label{michal}
S(x_0,r)\ni y\mapsto u(y)
\end{equation}
is close to the ellipse \eqref{ellipse} in the Sobolev norm. This implies
that the sum of the areas bounded by the projections of the curve \eqref{michal}
is close to the corresponding sum for the ellipse \eqref{ellipse} and hence is
different than zero. This, however, contradicts the fact that $f$ restricted 
to the circle $S(x_0,r)$ is horizontal. Now we will provide details
to support these claims.

The curves \eqref{ellipse} and \eqref{michal} are close in the Sobolev norm, but in dimension $1$ Sobolev functions are absolutely continuous, so actually the curves
are close in the supremum norm. This is a version of the Sobolev embedding theorem. We will provide a short proof
adapted to our particular situation. If
$$
g(y)=u(x_0) + du_{x_0}(y-x_0),
$$
then \eqref{e3} yields
\begin{equation}
\label{e5}
\barint_{S(x_0,r)} |u(z)-g(z)|\, d\sigma(z) <C\eps r.
\end{equation}
Moreover, the inequality \eqref{upper} applied to $u-g$ along with the inequality \eqref{e4} yield that for
any $y,z\in S(x_0,r)$
\begin{equation}
\label{e6}
|(u(y)-g(y))-(u(z)-g(z))| \leq
\int_{S(x_0,r)} |du_w-du_{x_0}|\, d\sigma(w) <C\eps r.
\end{equation}
Thus taking the average in \eqref{e6} with respect to 
$z\in S(x_0,r)$ and using \eqref{e5} we see that for 
any $y\in S(x_0,r)$ we have
\begin{eqnarray*}
|u(y)-g(y)| 
& \leq &
\barint_{S(x_0,r)} |(u(y)-g(y))-(u(z)-g(z))|\, d\sigma(z)\\ 
& + &
\barint_{S(x_0,r)} |u(z)-g(z)|\, d\sigma(z) <C\eps r,
\end{eqnarray*}
i.e.
\begin{equation}
\label{e7}
\sup_{y\in S(x_0,r)} |u(y)-g(y)| <C\eps r.
\end{equation}
The sum of oriented areas of projections of the ellipse $g(S(x_0,r))$
equals $C'r^2$ for some $C'>0$. 
The constant $C'$ depends only on the choice of $x_0\in A\subset\Omega$
and hence it does not depend on $\eps$.
Using Stokes' theorem
we can write this sum of areas as an integral over the circle
\begin{eqnarray}
\label{r1}
\lefteqn{C'r^2 
 = 
\frac{1}{2}\sum_{i=1}^n \int_{S(x_0,r)} g^{x_i}dg^{y_i} - g^{y_i}dg^{x_i}}\\
& = &
\frac{1}{2}\sum_{i=1}^n\int_{S(x_0,r)} (g^{x_i}-g^{x_i}(x_0))dg^{y_i} - (g^{y_i}-g^{y_i}(x_0))dg^{x_i}. \nonumber
\end{eqnarray}
On the other hand the curve $f$ restricted to the circle $S(x_0,r)$ is horizontal
and hence $u=\pi\circ f$ satisfies
\begin{eqnarray}
\label{r2}
\lefteqn{\sum_{i=1}^n \int_{S(x_0,r)} (u^{x_i}-u^{x_i}(x_0))du^{y_i} -(u^{y_i}-u^{y_i}(x_0))du^{x_i}} \\
& = &
\sum_{i=1}^n \int_{S(x_0,r)} u^{x_i}du^{y_i} -u^{y_i}du^{x_i} = 0. \nonumber
\end{eqnarray}
Subtracting \eqref{r2} from \eqref{r1} and using the fact that $g(x_0)=u(x_0)$ yields
\begin{eqnarray*}
2C'r^2 
& \leq &
\sum_{i=1}^n \left|\int_{S(x_0,r)} (g^{x_i}-g^{x_i}(x_0))dg^{y_i} -(u^{x_i}-u^{x_i}(x_0))du^{y_i}\right| \\
& + &
\sum_{i=1}^n \left|\int_{S(x_0,r)} (g^{y_i}-g^{y_i}(x_0))dg^{x_i} -(u^{y_i}-u^{y_i}(x_0))du^{x_i}\right| \\
& \leq &
C'' \left(\int_{S(x_0,r)} |g-g(x_0)|\, |dg-du| + \int_{S(x_0,r)} |g-u|\, |du| \right) \\
& < &
C'''\eps r^2.
\end{eqnarray*}
The estimate of the first integral in the last inequality 
follows from the fact that $|g(y)-g(x_0)|\leq Cr$ for $y\in S(x_0,r)$,
and the inequality \eqref{e4} (because $dg_y=du_{x_0}$), while in the estimate of
the second integral we used \eqref{e7} and \eqref{e44}. Taking $\eps>0$ sufficiently small
leads to a contradiction.
\hfill $\Box$

\section{Approximately Lipschitz mappings into the Heisenberg group}
\label{ih}

In this section we will prove Theorem~\ref{bd_lip} as well as Corollaries~ \ref{small image Sobolev} and \ref{super critical}. The proofs will 
use the theory of Sobolev mappings into metric spaces, which we describe first.
For more details on the approach 
presented here, see \cite{DHLT,hajlaszt}.

In some of results in this section, we assume that the domain $\Omega\subset\bbbr^k$ is bounded and has smooth boundary. We make these assumptions only to guarantee the validity of an appropriate Poincar\'e inequality, and that constant functions are integrable over $\Omega$; the argument provided here easily passes to a more general setting. 
 
Let $\Omega$ be an open set in $\bbbr^k$. The Sobolev space $W^{1,p}(\Omega;\ell^\infty)$
of functions with values in the Banach space $\ell^\infty$ of bounded sequences
can be defined with the notion of the Bochner
integral and weak derivatives. Every separable metric space $(X,d)$ can be isometrically embedded into $\ell^\infty$.
For example, one can use the well-known Kuratowski embedding. Let
$$
\kappa:X\to\ell^\infty
$$
be an isometric embedding. The Sobolev space of mappings with values into the metric space $X$ is defined 
as follows
\begin{equation}
\label{into_X}
W^{1,p}(\Omega;X)=\{ u:\Omega\to X:\, \kappa\circ u\in W^{1,p}(\Omega;\ell^\infty)\}.
\end{equation}
It turns out that this definition does not depend on the particular choice of the isometric embedding; the space can also be characterized in the intrinsic terms that do not refer to any embedding.
In particular the definition \eqref{into_X} can be used to define the space of Sobolev mappings into
the Heisenberg group, $W^{1,p}(\Omega,\bbbh^n)$. 
\begin{proposition}
\label{new}
Let $\Omega\subset\bbbr^k$ be a bounded domain with smooth boundary. For $1 \leq p < \infty$, if
$f\in W^{1,p}(\Omega,\bbbr^{2n+1})$ satisfies the weak contact equation
\eqref{Sob weak contact}, then $f\in W^{1,p}(\Omega,\bbbh^n)$.
\end{proposition}
\begin{remark} 
This result was proved in \cite[Proposition~6.8]{DHLT},
but under the additional assumption that $f$ is bounded.
\end{remark}
\begin{proof}
Let $\kappa:\bbbh^n\to\ell^\infty$ be an isometric embedding. It suffices to
show that $\kappa\circ f\in W^{1,p}(\Omega;\ell^\infty)$.

We begin by showing that $\kappa \circ f \in L^{p}(\Omega;\ell^\infty)$. 
Since $\Omega$ is bounded and $\kappa$ is an isometry, this will follow once we have shown that
$$
\int_{\Omega} \Vert f(x)\Vert^p_{\rm{K}} \ d\H_{\bbbr^k}^k(x)<\infty.
$$
This, in turn, follows from the assumption that $f \in L^p(\Omega;\bbbr^{2n+1})$, since there exists a number $C \geq 1$, depending only on $n$, such that for any $p \in \bbbh^n$,
$$
\Vert p\Vert_{\rm{K}} \leq C\max\{1,\Vert p\Vert_{\bbbr^{2n+1}}\}.
$$
We now show that $\kappa \circ f$ has $p$-integrable weak partial derivatives. The mapping $f$ is absolutely continuous on almost all line
segments $\ell:[0,L]\to\Omega$ parameterized by the arc length that are parallel to coordinate directions. Since
$f$ satisfies the contact equation, 
$$
\gamma=(\gamma_{x_1},\gamma_{y_1},\ldots,\gamma_{x_n},\gamma_{y_n},\gamma_t)=
f\circ\ell:[0,L]\to\bbbh^n
$$
is horizontal for almost all such line segments $\ell$. Fix such a segment.
Recall that the Carnot-Carath\'eodory length $\ell_{\bbbh}(\gamma)$ is no greater than 
the Euclidean length of $\gamma$ (see Section~\ref{heisenberg}).
Hence for any pair of points $t_1<t_2$ in $[0,L]$,
\begin{equation}
\label{907}
d_{cc}(\gamma(t_1),\gamma(t_2))\leq \int_{t_1}^{t_2} \Vert \gamma'(s)\Vert_H\, ds \leq
\int_{t_1}^{t_2}\Vert \gamma'(s)\Vert_{\bbbr^{2n+1}}\, ds.
\end{equation}
Since $\kappa$ is an isometric embedding, this
implies that the curve $\kappa\circ\gamma:[0,L]\to\ell^\infty$ is absolutely continuous
and so the $w^*$-derivative $(\kappa\circ\gamma)'(s)\in\ell^\infty$ exists at almost all 
$s\in [0,L]$, \cite[Lemma~2.8]{hajlaszt}.
The $w^*$-derivative is a $w^*$-limit of difference quotients, hence it follows from \eqref{907}
that 
$$
\Vert (\kappa\circ\gamma)'(s)\Vert_{\ell^\infty}\leq \Vert \gamma'(s)\Vert_{\bbbr^{2n+1}}
$$
almost everywhere. This means the $w^*$-partial derivatives of $\kappa\circ f$ exist
a.e.\ in $\Omega$ and they are bounded by the Euclidean weak partial derivatives of 
$f:\Omega\to\bbbr^{2n+1}$. 
Hence Lemma~2.12 in \cite{hajlaszt} yields that
$f\in W^{1,p}(\Omega,\bbbh^n)$. 
\end{proof}

A stronger version of the following approximation result was proved in
\cite[Proposition~5.4]{DHLT}, \cite[Proposition~1.2]{hajlasz3}.
\begin{proposition} 
Let $\Omega \subset \bbbr^k$ be a bounded domain with smooth boundary.  If $u\in W^{1,1}(\Omega;\ell^\infty)$, 
then for every $\eps>0$ there is a Lipschitz mapping $g\in\lip(\Omega,\ell^\infty)$
such that $\H^k(\{x\in\Omega:\, u(x)\neq g(x)\})<\eps$.
\end{proposition}

\begin{proof}[Proof of Theorem~\ref{bd_lip}] 
Let $\Omega \subeq \bbbr^k$ be a bounded domain with smooth boundary, and let $f \in W^{1,1}(\Omega;\bbbr^{2n+1})$ satisfy the weak contact equation~\eqref{Sob weak contact}. According to Proposition~\ref{new},
$f\in W^{1,1}(\Omega,\bbbh^n)$.
Since $\kappa\circ f\in W^{1,1}(\Omega;\ell^\infty)$, for any $\eps>0$ there is $g\in\lip(\Omega;\ell^\infty)$
such that $\H^k(\{x\in\Omega:\, \kappa\circ f(x)\neq g(x)\})<\eps$. Denote  
$E_\eps=\{x\in\Omega:\, \kappa\circ f(x)=g(x)\}$.
Thus $\H^k(\Omega\setminus E_\eps)<\eps$.
Clearly $g|_{E_\eps}:E_\eps\to \kappa(\bbbh^n)$. Hence $\kappa^{-1}\circ g|_{E_\eps}\in \lip(\Omega;\bbbh^n)$
and it coincides with $f$ on $E_\eps$. 
\end{proof}

\begin{proof}[Proof of Corollary~\ref{small image Sobolev}] Let $k>n$, let $\Omega \subeq \bbbr^k$ be an open set, and let $f \in W^{1,1}_{\loc}(\Omega; \bbbr^{2n+1})$ satisfy the weak contact equation~\eqref{Sob weak contact}. Noting that any open subset of $\bbbr^k$ can be exhausted by countably many balls, the countable subadditivity of Hausdorff measure allows us to assume that $\Omega$ is a ball, and that $f$ and $\wk df$ are integrable on $\Omega$. Theorem~\ref{bd_lip} implies that there is a sequence of subsets $E_1 \supseteq E_2 \supseteq \hdots$ of $\Omega$ such that $f|_{\Omega \backslash E_i}$ is Lipschitz and $Z=\bigcap E_i$ satisfies
$$\H^k_{\bbbr^k}\left(Z\right) = 0.$$
By Theorem \ref{pure k unrect} and the countable subadditivity of Hausdorff measure, 
$$\H^k_{\bbbh^n}\left(f(\Omega\backslash Z)\right)=0.$$
The Lusin condition \eqref{Nkk} now implies that $\H^k_{\bbbh^n}\left(f(\Omega)\right)=0,$ as desired. \end{proof}

\begin{proof}[Proof of Corollary~\ref{super critical}]
As above, we can assume that $\Omega \subseteq \bbbr^k$ is bounded with smooth boundary, and that $f\in W^{1,p}(\Omega;\bbbr^{2n+1})$ satisfies the weak contact equation~\eqref{Sob weak contact}.  According to Proposition~\ref{new},
$f\in W^{1,p}(\Omega,\bbbh^n)$ and hence \cite[Theorem~1.3]{EucPeano} implies that $f$ has the Lusin property \eqref{Nkk}.
Now the result follows from Corollary~\ref{small image Sobolev}.
\end{proof}

\section{Unrectifiability}
\label{main}

Having already proven Theorem~\ref{low rank lipschitz}, the remaining portion of our
proof of Theorem~\ref{pure k unrect} is related to that of Sard's theorem~\cite{sternberg}.
Let $\vi:\bbbr^M\to\bbbr^N$ be sufficiently smooth. 
In the proof of the Sard theorem one shows first that the image of the set of points where the rank of the derivative is zero has zero Hausdorff measure in the appropriate dimension, which depends on the smoothness of $\vi$. 
Then, for each number $r$ less than the maximal rank of the derivative of $\vi$, one obtains a similar estimate for the image of the set where the rank of the derivative equals $r$ by reducing the problem to the case of zero rank. Namely, using a suitable change of variables (related to the implicit function theorem), one can assume that $\vi$ restricted to the first $(M-r)$ coordinates of $\bbbr^M$ has rank zero.
Estimates depend on the smoothness of mapping, but they are also available in the $C^1$ case.

In our situation, the rank of the approximate derivative is at most $n$ almost everywhere. Although the mapping is not of class $C^1$, it coincides with a $C^1$ mapping $g$ on a set that is arbitrarily 
large in measure. We will apply the change of variables to the mapping $g$. This will
reduce the problem to the case of rank zero. Combining it 
with a careful investigation
of the geometry of the Heisenberg group will imply that the $k$-dimensional Hausdorff measure
(with respect to the Carnot-Carath\'eodory metric) of the image equals zero.

Let $k>n$, and let $f:E\to\bbbh^n$, $E\subset\bbbr^k$, be a locally Lipschitz mapping. If $A\subset E$ has measure zero,
then $\H^k_{\bbbh^n}(f(A))=0$. Thus it suffices to prove that there is a full measure subset
of $E$ whose image under $f$ has zero $k$-dimensional measure.
As discussed in Section \ref{heisenberg}, the mapping $f$ is  
approximately differentiable at almost all points of $E$.
so it coincides with $C^1$ mappings on sets large in measure.
Let $g:\bbbr^k\to\bbbr^{2n+1}$ be a mapping
of class $C^1$ which agrees with $f$ on a set $K\subset E$ and $\ap df=dg$ in $K$.
For $j\in \{0,1,2,\ldots,n\}$ let 
$$
K_j=\{x\in K:\, \rank dg_x=j\}.
$$

According to Theorem~\ref{low rank lipschitz}, the rank of the derivative of $f$ is bounded by $n$ almost 
everywhere and hence
$$
\H^k\left(K\setminus \bigcup_{j=0}^n K_j\right)=0.
$$
It suffices to prove that $\H^k_{\bbbh^n}(f(K_j))=0$ for all $j=0,1,2,\ldots,n$, because the set $K$ can be chosen so that $E\backslash K$ has arbitrarily small measure. Moreover, by removing 
a subset of measure zero from $K_j$ we can assume that all points of $K_j$ are density points
and that the image of $d(\pi\circ g)$ is isotropic on $K_j$ (see Remark~\ref{rem}).

To prove that $\H^k(f(K_0))=0$ we do not need to make any change of variables, but if $j\geq 1$ we need to 
make a change of variables to reduce the problem to the case that $j=0$.
\begin{lemma}
\label{change}
Let $x_0\in K_j$, $1\leq j\leq n$. Then there is a neighborhood $U$ of $x_0$, a diffeomorphism
$\Phi:U\to\bbbr^k$, and an affine isometry
$\Psi:\bbbh^n\to\bbbh^n$ such that
\begin{itemize}
\item $\Phi^{-1}(0)=x_0$ and $\Psi(g(x_0))=0$;
\item there is $\eps>0$ such that for 
$p=(p_1,\ldots,p_k)\in B_{\bbbr^k}(0,\eps)$ and $i=1,2,\ldots,j$
$$
\left(\Psi\circ g\circ\Phi^{-1}(p)\right)^{x_i}=p_i.
$$
\end{itemize}
\end{lemma}
{\em Proof.}
By pre-composing $g$ with an Euclidean translation and post-composing $g$ with a
Heisenberg translation, we may assume without loss of generality that 
$x_0=0\in\bbbr^k$ and $g(0)=0\in\bbbr^{2n+1}$.
Since the horizontal space at $0\in\bbbh^n$ is $\bbbr^{2n}\times \{0\}$,
the image of $dg_0$ is an isotropic subspace of $\bbbr^{2n}\times \{0\}$.
By Lemma~\ref{standard isotropic} there is a linear isometry $\Psi:\bbbh^n\to\bbbh^n$ such that
$$
{\rm im}\,  d(\Psi\circ g)_0 = {\rm span}\, 
\left\{ \left.\frac{\partial}{\partial x_1}\right|_0,\ldots,\left.\frac{\partial}{\partial x_j}\right|_0\right\}\, .
$$
Now we can find  a linear isomorphism $\alpha:\bbbr^k\to\bbbr^k$ such that
$$
d(\Psi \circ g \circ \alpha^{-1})_0\left(\frac{\partial}{\partial p_i}\right)
= \begin{cases}
			\frac{\partial}{\partial x_i} & \mbox{if  $1\leq i \leq j$},\\
					0 & \mbox{if $j < i \leq k$}.\
	\end{cases}
$$
Define $\beta:\bbbr^k\to\bbbr^k$ by
$$
\beta(p) = 
\left((\Psi \circ g \circ \alpha^{-1}(p))^{x_1},\hdots,(\Psi \circ g \circ \alpha^{-1}(p))^{x_j},p_{j+1},\hdots,p_k\right)\, .
$$
It is easy to see that the matrix of the derivative $d\beta_0$ is the $k$-dimensional identity matrix,
so $\beta$ is a diffeomorphism in a neighborhood of $0\in\bbbr^k$,
$\beta:B_{\bbbr^k}(0,\eps)\to U=\beta(B_{\bbbr^k}(0,\eps))$.
Now $\Phi=\beta\circ\alpha$ satisfies the claim of the lemma.
\hfill $\Box$

In what follows, all cubes will have edges parallel to coordinate axes.
By the countable additivity of the Hausdorff measure, it suffices to show that every point in $K_j$ has a neighborhood whose intersection 
with $K_j$ is mapped onto a set of $\H^k_{\bbbh^n}$-measure zero. Thus, by Lemma~\ref{change},
we may assume without loss of generality that mapping $g$ satisfies $g^{x_i}(p)=p_i$ for $i\leq j$ and that the set $K_j$
has small diameter, say the closure of
$K_j$ is contained in the interior of the cube $[0,1]^k$.
We may also assume that $f$ is $L$-Lipschitz. Since $\rank dg =j$ on $K_j$ and $g$
fixes the first $j$ coordinates, the derivative of $g$ in directions orthogonal to the first
$j$ coordinates equals zero at the points of $K_j$.

Now, the rough idea is as follows. Choose a small cube around a point in $K_j$, say
$[0,d]^k=[0,d]^j\times [0,d]^{k-j}$. 
For a large positive integer $m$, divide $[0,d]^j$ into $m^j$ small cubes $\{Q_\nu\}_{\nu=1}^{m^j}$, each of edge-length $dm^{-1}$. This will split the cube $[0,d]^k$ into thin and tall rectangular boxes
$Q_\nu\times [0,d]^{k-j}$. The mapping $g$ maps such a tall box into $Q_\nu\times\bbbr^{2n+1-j}$, because it fixes
the first $j$ coordinates. However, the mapping $g$ in the directions 
orthogonal to $Q_\nu$ has rank zero on a large
subset. Hence the function $g$ grows slowly in these directions, and so each tall box will be squeezed
so that its image will be contained in a Kor\'anyi ball of radius $CLdm^{-1}$. More precisely, we shall prove:
\begin{lemma}
\label{sq}
There is a constant $C$, depending only on $k$, such that for any integer $m\geq 1$
and every $x\in K_j$ there is a closed cube $Q\ni x$ of an arbitrarily small edge-length
$d$ such that $f(K_j\cap Q)=g(K_j\cap Q)$ can be covered by $m^j$ Kor\'anyi-balls in $\bbbh^n$, each of
radius $CLdm^{-1}$.
\end{lemma}
The theorem easily follows from the lemma. Indeed, given $m\geq 1$, the family of
cubes described in Lemma \ref{sq} forms a Vitali covering of $K_j$
and hence by the Vitali covering theorem \cite[Theorem~II.17.1]{di}
we can select cubes $\{Q_i\}_{i=1}^\infty$ with edges of length $d_i$ 
and pairwise disjoint interiors such that
\begin{equation}
\label{1234}
\H^k\left(K_j\setminus \bigcup_{i=1}^\infty Q_i\right)=0.
\end{equation}
We may also assume that $\sum_{i=1}^\infty d_i^k\leq 1$, because
we may choose cubes $Q_i$ to be inside the unit cube that contains $K_j$.

Recall the definition of the Hausdorff content. In any metric space it is defined by 
$$
\H^k_\infty(A)=\inf\left\{ \sum_{i=1}^\infty r_i^k\right\}
$$
where the infimum is taken over all coverings 
$A\subset\bigcup_{i=1}^\infty B(x_i,r_i)$. It is easy to see that
$\H^k_\infty(A)=0$ if and only if $\H^k(A)=0$. 

According to Lemma \ref{sq},
$$
f(K_j\cap Q_i)\subset \bigcup_{\nu=1}^{m^j}B_{\bbbh^n}(x_{i\nu},CLd_im^{-1}).
$$
Hence
$$
\H^k_{\infty,\bbbh^n}\left(f\left(K_j\cap\bigcup_{i=1}^\infty Q_i\right)\right)\leq
\sum_{i=1}^\infty m^j(CLd_i m^{-1})^k \leq
C^k L^k m^{j-k}.
$$
Since $j-k<0$ and $m$ can be arbitrarily large, the estimate $C^k L^k m^{j-k}$
can be arbitrarily close to zero. This proves that for any $\eps>0$ there is a subset of
$K_j$ of full measure whose image has Hausdorff content less than $\eps$. This implies that
$\H^k_{\bbbh^n}(f(K_j))=0$, as desired.

Thus we are left with the proof of Lemma~\ref{sq}. 

\subsection{Proof of Lemma~\ref{sq}.} 
Fix a positive integer $m$. Let $x\in K_j$.
Since every point in $K_j$ is a density point, there is 
a closed cube $Q\ni x$ such that
\begin{equation}
\label{less_eps}
\H^k(Q\setminus K_j) < m^{-k}\H^k(Q).
\end{equation}
The cube $Q$ may be chosen so that it edge length $d>0$ is arbitrarily small. By translating the 
coordinate system in $\bbbr^k$ we may assume that 
$$
Q=[0,d]^k=[0,d]^j\times [0,d]^{k-j}.
$$
Divide $[0,d]^j$ into $m^j$ essentially disjoint cubes, each of edge-length  $dm^{-1}$. Denote the 
resulting cubes by $\{ Q_\nu\}_{\nu=1}^{m^j}$. Then
$$
Q=\bigcup_{\nu=1}^{m^j} Q_\nu\times [0,d]^{k-j}.
$$
Since the mapping $g$ fixes the first $j$ coordinates,
and $f=g$ on $K_j$,
$$
f((Q_\nu\times [0,d]^{k-j})\cap K_j)\subset g(Q_\nu\times [0,d]^{k-j}) \subset Q_\nu\times\bbbr^{2n+1-j}.
$$
It remains to prove that the above image of the function $f$ is contained in a Kor\'anyi-ball 
of radius $CLdm^{-1}$. Where the value of the quantity $C$, 
which depends on $k$ only, can be deduced from the estimates below. Fix $\nu$.
It follows from \eqref{less_eps} that
$$
\H^k ((Q_\nu\times [0,d]^{k-j})\cap K_j) > (m^{-j}-m^{-k})d^k.
$$
Fubini's theorem now yields that there is $\rho\in Q_\nu$ such that
\begin{equation}
\label{star}
\H^{k-j}((\{\rho\}\times [0,d]^{k-j})\cap K_j)>(1-m^{j-k}) d^{k-j}.
\end{equation}
We will show that the diameter of the set
$f((\{\rho\}\times [0,d]^{k-j})\cap K_j)$ 
in the Kor\'anyi metric
is bounded by $CLdm^{-1}$.
This easily implies the lemma. Indeed, it follows from \eqref{star}
that the distance of any point in the slice 
$\{\rho\}\times [0,d]^{k-j}$ to the set
$(\{\rho\}\times [0,d]^{k-j})\cap K_j$ is bounded by $Cdm^{-1}$.
Since the distance of any $p\in (Q_\nu\times [0,d]^{k-j})\cap K_j$
to the slice $\{\rho\}\times [0,d]^{k-j}$ is also bounded by $Cdm^{-1}$
we conclude that
$$
\dist_{\bbbr^k}\left( p,(\{\rho\}\times [0,d]^{k-j})\cap K_j\right) \leq Cdm^{-1}.
$$
Now the fact that $f$ is $L$-Lipschitz on $K_j$ shows that the Kor\'anyi diameter
of the set
$f((Q_\nu\times [0,d]^{k-j})\cap K_j)$ is bounded by $CLdm^{-1}$ plus
the diameter of the set
$f((\{\rho\}\times [0,d]^{k-j})\cap K_j)$, which is also 
bounded by $CLdm^{-1}$. Hence the set
$f((Q_\nu\times [0,d]^{k-j})\cap K_j)$ is contained in a ball of radius
$CLdm^{-1}$, proving the lemma.

Thus we are left with the proof that
\begin{equation}
\label{last piece}
\diam_{\bbbh^n}\left( f((\{\rho\}\times [0,d]^{k-j})\cap K_j)\right) \leq CLdm^{-1}.
\end{equation}
For $j+1\leq i\leq k$ let
$$
F^i=\{ (\rho,p_{j+1},\ldots,p_k)\in \{\rho\}\times [0,d]^{k-j}:\, p_i=0\}
$$
be one of the faces of the cube $\{\rho\}\times [0,d]^{k-j}$. 
For 
$$
p=(\rho,p_{j+1},\ldots,p_{i-1},0,p_{i+1},\ldots,p_k)\in F^i
$$
let $\ell^i_p$ be the segment of length $d$ in $\{\rho\}\times [0,d]^{k-j}$,
perpendicular to the face $F^i$ and passing through $p$. In other
words the segment $\ell^i_{p}$ is the image of the parameterization
$$
t\mapsto p+e_it =(\rho,p_{j+1},\ldots,p_{i-1},t,p_{i+1},\ldots,p_k),
\quad
t\in [0,d].
$$
Let
\begin{equation}
\label{oko}
\tilde{F}^i = \{p\in F^i:\, \H^1(\ell^i_{p}\cap K_j)>(1-m^{-1})d\}.
\end{equation}
It easily follows from \eqref{star} and the Fubini theorem that
\begin{equation}
\label{star2}
\H^{k-j-1}(\tilde{F}^i)>(1-m^{j-k+1})d^{k-j-1}.
\end{equation}
If $j+1=k$, this simply means that $F^i=\tilde{F}^i$ and $i=k$.
\begin{lemma}
\label{short}
For any $p \in F_i$,
$$
\diam_{\bbbh^n}(f(K_j\cap\ell^i_{p}))\leq L \H^1(\ell^i_{p}\setminus K_j).
$$
In particular, if $\tilde{p}\in \tilde{F}^i$, then
$$
\diam_{\bbbh^n}(f(K_j\cap\ell^i_{\tilde{p}}))\leq Ldm^{-1}.
$$
\end{lemma}

{\em Proof.} According to Lemma~\ref{line extension Lipschitz}, $f$ restricted to
$K_j\cap\ell^i_{p}$ can be extended to $\ell^i_{p}$
as an $L$-Lipschitz curve into $\bbbh^n$.  
The extension of $f$ to $\ell^i_{p}$ is a horizontal curve 
and hence its length in the metric 
of $\bbbh^n$ equals the integral of speed computed with respect to metric $\boldg$
in the horizontal space. The speed is always bounded by the Lipschitz constant $L$.
However, on the set $K_j\cap\ell^i_{p}$, $f$ coincides with $g$. Hence the derivatives
of $f$ and $g$ in the direction of the segment coincide at almost all points
of  $K_j\cap\ell^i_{p}$. However, the derivative of $g$ in this direction equals zero and so the speed of $f$ is zero at almost every point of $K_j \cap \ell^i_p$. Hence the diameter we wish to estimate is bounded by the integral of the speed over the set $\ell^i_{p}\setminus K_j$, where the speed is bounded by $L$. \hfill $\Box$

If $p\in F^i$, it follows from \eqref{star2} that there is $\tilde{p}\in \tilde{F}^i$
such that $|p-\tilde{p}|<Cdm^{-1}$. This inequality and \eqref{oko} imply that for any point 
$q\in\ell^i_{p}$ there is a point $\tilde{q}\in\ell^i_{\tilde{p}}\cap K_j$
such that $$|q-\tilde{q}|<(C+1)dm^{-1}=Cdm^{-1}.$$

We are now prepared to prove \eqref{last piece} and complete the proof of Theorem~\ref{pure k unrect}. Let $a,b\in (\{\rho\}\times [0,d]^{k-j})\cap K_j$. We can connect the points $a$ and $b$
by $k-j$ segments $I^i$ parallel to coordinate directions $e_i$, $i=j+1,\ldots,k$,
so that the total length of these segments is bounded by $C|a-b|$. 
Some of the segments might degenerate to a point if the corresponding coordinates of
$a$ and $b$ are equal. Denote the endpoints
of $I^i$ by $\alpha_i,\beta_i$. We have $\alpha_{j+1}=a$, $\beta_i=\alpha_{i+1}$, $\beta_{k}=b$ and
$$
\sum_{i=j+1}^{k} |\alpha_i-\beta_i| <C|a-b|.
$$
Segments $I^i$ do not necessarily have large intersections with $K_j$, so in order to apply 
Lemma~\ref{short} we need to shift these segments slightly. This is what we will do now.
Each segment $I^i$ is contained in a segment $\ell^i_{p_i}$, $p_i\in F^i$,
so $\alpha_i,\beta_i\in \ell^i_{p_i}$.
Let $\tilde{p}_i\in \tilde{F}^i$ be such that $|p_i-\tilde{p}_i|<Cdm^{-1}$.
Thus there are points $\tilde{\alpha}_i,\tilde{\beta}_i\in \ell^i_{\tilde{p}_i}\cap K_j$ such that
$|\alpha_i-\tilde{\alpha}_i|<Cdm^{-1}$ and $|\beta_i-\tilde{\beta}_i|<Cdm^{-1}$.
Note that it is not necessarily true that $\tilde{\beta}_i=\tilde{\alpha}_{i+1}$, but nevertheless 
we have the estimate $|\tilde{\beta}_i-\tilde{\alpha}_{i+1}|<Cdm^{-1}$.
It follows from Lemma~\ref{short} that 
$$
d_{\rm{K}}(f(\tilde{\alpha}_i),f(\tilde{\beta}_i))< Ldm^{-1}.
$$
Hence
\begin{eqnarray*}
\lefteqn{d_{\rm{K}}(f(a),f(b))
 = 
d_{\rm{K}}(f(\alpha_{j+1}),f(\beta_{k}))
\leq 
d_{\rm{K}}(f(\alpha_{j+1}),f(\tilde{\alpha}_{j+1}))} \\
& + &
\sum_{i=j+1}^{k-1} 
     \left( d_{\rm{K}}(f(\tilde{\alpha}_i),f(\tilde{\beta}_i)) + 
           d_{\rm{K}}(f(\tilde{\beta}_i),f(\tilde{\alpha}_{i+1}))\right) \\
& + &
d_{\rm{K}}(f(\tilde{\alpha}_k),f(\tilde{\beta}_k)) +
d_{\rm{K}}(f(\tilde{\beta}_k),f(\beta_k))   \\
& \leq &        
CLdm^{-1}.
\end{eqnarray*}
This proves the desired estimate \eqref{last piece}.
The proof is complete.
\hfill $\Box$

\section{Non-existence of H\"older-Lipschitz embeddings}
\label{gromov conjecture}

In this section we will prove Theorem~\ref{Lip Gromov}. To do this we will need two
auxiliary results.
\begin{proposition}
\label{contact}
Let $k$ and $n$ be arbitrary positive integers and let $\Omega\subset\bbbr^k$ be open. Suppose that
$f:\Omega\to\bbbh^n$ is of class $C^{0,\frac{1}{2}+}(\Omega;\bbbh^n)$. If the components
$f^{x_i}, f^{y_i}$ are differentiable at $x_0\in\Omega$ for $i=1,2,\ldots,n$, then
the last component $f^t$ is also differentiable at $x_0$ and
$$
df^t_{x_0} = 2\sum_{i=1}^n
\left( f^{y_i}(x_0)df_{x_0}^{x_i}-f^{x_i}(x_0)df_{x_0}^{y_i}\right)\, ,
$$
i.e. the image of the derivative $df_{x_0}$ lies in the horizontal space $H_{f(x_0)}\bbbh^n$.
\end{proposition}
{\em Proof.}
From the assumptions about $f$, there is a continuous, non-decreasing function 
$\beta:[0,\infty)\to [0,\infty)$ satisfying $\beta(0)=0$ such that for all $x \in \Omega$
$$
d_{\rm{K}}(f(x),f(x_0))\leq |x-x_0|^{1/2}\beta(|x-x_0|).
$$
Hence \eqref{kor} yields
\begin{eqnarray*}
\lefteqn{\left|f^t(x)-f^t(x_0) + 2\sum_{i=1}^n
\left( f^{x_i}(x_0)f^{y_i}(x) - f^{x_i}(x)f^{y_i}(x_0)\right)\right|^{1/2}} \phantom{aaaaaaaaaaaaa}\\ 
& \leq &
|x-x_0|^{1/2}\beta(|x-x_0|).
\end{eqnarray*}
After adding and subtracting $f^{x_i}(x_0)f^{y_i}(x_0)$ in the sum given above, the triangle inequality implies 
\begin{eqnarray*}
\lefteqn{\left| f^t(x)-f^t(x_0) - 
2\sum_{i=1}^n \left(f^{y_i}(x_0)df_{x_0}^{x_i} - f^{x_i}(x_0)df_{x_0}^{y_i}\right)(x-x_0)\right|} \\
& \leq &
|x-x_0|\beta^2(|x-x_0|) \\ 
& + &
2\sum_{i=1}^n |f^{x_i}(x_0)|\, |f^{y_i}(x)-f^{y_i}(x_0)-df_{x_0}^{y_i}(x-x_0)|\\
& + &
2\sum_{i=1}^n |f^{y_i}(x_0)|\, |f^{x_i}(x)-f^{x_i}(x_0)-df_{x_0}^{x_i}(x-x_0)|\\
& = &
o(|x-x_0|),
\end{eqnarray*}
because the functions $f^{x_i}$ and $f^{y_i}$ are differentiable at $x_0$. This implies the conclusion of the proposition.
\hfill $\Box$

\begin{proposition}
\label{get Heis Lip}
Let $k$ and $n$ be arbitrary positive integers.
Suppose that
$f:[0,1]^k\to \bbbr^{2n+1}$ is Lipschitz and at $\H^k$-almost every point $x_0$, the image of the derivative $df_{x_0}$ lies in the horizontal space $H_{f(x_0)}\bbbh^n$.
Then $f:[0,1]^k\to\bbbh^n$ is Lipschitz.
\end{proposition}
{\em Proof.}
It follows from the Fubini theorem that almost all segments parallel
to coordinate axes are mapped by $f$ onto horizontal curves. 
The Euclidean speed on these curves if bounded by the Lipschitz
constant $L$ of $f$. The image
$f([0,1]^k)$ is a bounded subset of $\bbbh^n$. On bounded subsets of $\bbbh^n$
the Euclidean length of horizontal vectors is uniformly comparable 
to the length computed with respect to the sub-Riemannian metric $\boldg$ in the horizontal
distribution. Thus the images of almost all segments parallel to coordinate axes
are $CL$-Lipschitz as curves in $\bbbh^n$. 
Any segment parallel to a coordinate axis is a limit of parallel segments
on which $f$ is $CL$-Lipschitz as a mapping into $\bbbh^n$. Hence the 
mapping $f$ on that segment is also $CL$-Lipschitz as a uniform limit of
$CL$-Lipschitz functions. Thus $f$ is $CL$-Lipschitz on {\em all}
segments parallel to coordinates. Hence $f:[0,1]^k\to\bbbh^n$
is $CL$-Lipschitz.
\hfill $\Box$

\subsection{Proof of Theorem~\ref{Lip Gromov}}
We consider an open set $\Omega\subset\bbbr^k$, $k>n$.
Suppose that $f\in C^{0,\frac{1}{2}+}(\Omega;\bbbh^n)$
is locally Lipschitz as a mapping into $\bbbr^{2n+1}$.
By taking a subset of $\Omega$ we may assume that
$\Omega$ is a cube and that $f$ is Lipschitz.
It follows from Rademacher's theorem and from Proposition~\ref{contact}
that image of the classical derivative of $f$ is in the horizontal distribution at almost every point. Hence $f:\Omega\to\bbbh^n$
is Lipschitz by Proposition~\ref{get Heis Lip}.
Now Theorem~\ref{pure k unrect} implies that
$\H^k_{\bbbh^n}(f(\Omega))=0$. 
Since the identity mapping from $\bbbh^n$ to $\bbbr^{2n+1}$
is Lipschitz on compact sets, we also see that
$H^k_{\bbbr^{2n+1}}(f(\Omega))=0$.
This implies that the topological dimension of $f(\Omega)$ is
at most $k-1$, \cite[Theorem~8.15]{heinonen}.
Since the topological dimension is invariant under homeomorphisms,
$f$ cannot be injective on $\Omega$, as otherwise the image would have topological dimension $k$.
\hfill $\Box$

\end{document}